\documentclass[table,11pt,a4paper]{scrartcl}
\usepackage{etoolbox}
\newtoggle{anonymous}
\togglefalse{anonymous}
\usepackage{quoting}
\quotingsetup{vskip=4pt}

\usepackage{amssymb,amsmath,amsthm,fullpage,url}
\usepackage{enumerate} 
\usepackage{complexity}
\usepackage{thm-restate}
\usepackage{graphics}
\usepackage{multirow,mathtools,booktabs} 
\usepackage[dvipsnames,svgnames,table]{xcolor}

\usepackage[colorlinks=true,linkcolor=blue,citecolor=ForestGreen]{hyperref} 
\usepackage{csquotes}
\usepackage{orcidlink}
\usepackage{setspace}
\usepackage{subcaption}
\usepackage{tikz}
\usetikzlibrary{arrows}
\usepackage[capitalise]{cleveref}

\newcommand{\cq}{\coloneqq}
\newcommand{\ts}{\textstyle}

\newcommand{\piw}{\pi_{w}}
\newcommand{\good}{R}

\newcommand{\vex}{\ensuremath{\alpha}}
\newcommand{\biasname}{\ensuremath{\theta}}

\newcommand{\cups}{\cup \cdots \cup}

\newcommand{\Quad}[1]{\quad\text{#1}\quad}
\newcommand{\Qfor}{\Quad{for}}
\newcommand{\Qforall}{\Quad{for all}}
\newcommand{\Qwith}{\Quad{with}}
\newcommand{\Qwhere}{\Quad{where}}
\newcommand{\Qand}{\Quad{and}}

\newcommand{\ceil}[1]{\lceil #1 \rceil}
\newcommand{\floor}[1]{\lfloor #1 \rfloor}

\newcommand{\mcb}{\mathcal B}
\newcommand{\mbn}{\mathbb N}
\newcommand{\mbpn}{\mathbb N \setminus \{0\}}

\usepackage[normalem]{ulem}
\newcommand{\SuggestChange}[2]{%
	{\color{red} \relax\ifmmode\text{\sout{$#1$}}\else \sout{#1}\fi } \ {\color{blue} #2 } %
}

\DeclareMathOperator{\srw}{SRW}
\DeclareMathOperator{\brw}{BRW}
\DeclareMathOperator{\tbrw}{TBRW}

\newcommand{\dmax}{d_{\mathsf{max}}}

\newcommand{\dmin}{d_{\mathsf{min}}}
\newcommand{\ct}{\mathcal{T}}

\renewcommand{\Pr}[1]{\mathbb{P}\left[\,#1\,\right]}

\newcommand{\Ex}[1]{\mathbb{E} \left[\,#1\,\right]}

\newcommand*{\abs}[1]{\lvert #1\rvert}

\newcommand{\ETBcov}[2]{C_{#1}^{\eps\mathsf{TB}}(#2)}

\newcommand{\tetb}{t_{\mathsf{cov}}^{\eps\mathsf{TB}}}
\newcommand{\tcrw}{t_{\mathsf{cov}}^{\mathsf{CRW}}}

\newcommand{\tmix}{t_{\mathsf{mix}}}
\newcommand{\thit}{t_{\mathsf{hit}}}

\newcommand{\dist}{\operatorname{dist}}

\renewcommand{\leq}{\leqslant}
\renewcommand{\geq}{\geqslant}
\renewcommand{\le}{\leqslant}
\renewcommand{\ge}{\geqslant}
\renewcommand{\tilde}{\widetilde}

\newcommand{\eps}{\varepsilon}

\allowdisplaybreaks

\makeatletter
\def\th@plain{%
	\thm@headfont{\bfseries\sffamily}
	\thm@notefont{\mdseries\rmfamily}
	\slshape 
}
\def\th@definition{%
	\thm@headfont{\bfseries\sffamily}
	\thm@notefont{\mdseries\rmfamily}
}
\makeatother

\newenvironment{Proof}[1][\proofname]{%
	\proof[\upshape\bfseries\sffamily\boldmath#1]
}{\endproof}

\theoremstyle{plain}
\newtheorem{theorem}{Theorem}[section]
\newtheorem{lemma}[theorem]{Lemma}
\newtheorem{corollary}[theorem]{Corollary}

\newtheorem{proposition}[theorem]{Proposition}

\counterwithin{figure}{section}
\counterwithin{equation}{section}

\newtheorem{pbl}{Open Problem}
\newtheorem{que}{Question}

\theoremstyle{definition}
\newtheorem{definition}[theorem]{Definition}

\crefname{figure}{Figure}{Figures}

\crefformat{equation}{(#2#1#3)}
\Crefformat{equation}{(#2#1#3)}

\newenvironment{intcom}{\par\medskip
	\noindent \textbf{Internal comment} \rmfamily}{\medskip}

\newif\ifhideinternalcomments
\hideinternalcommentstrue 

\ifhideinternalcomments
\usepackage{environ}
\NewEnviron{hide}{}

\fi

\title{
	Time-Biased Random Walks and  \\
	Robustness of Expanders
}
\date{} 
\iftoggle{anonymous}{
	\author{ }
}{%
	\author{Sam Olesker-Taylor\thanks{Department of Statistics, University of Warwick, UK, \texttt{oleskertaylor.sam@gmail.com}, \orcidlink{0000-0001-9764-1645}} \and Thomas Sauerwald\thanks{Department of Computer Science \& Technology, University of Cambridge, UK, \texttt{tms41@cam.ac.uk}, \orcidlink{0000-0002-0882-283X}}\and John Sylvester\thanks{Department of Computer Science, University of Liverpool, UK, \texttt{john.sylvester@liverpool.ac.uk}, \orcidlink{0000-0002-6543-2934}}}
}
\begin{document}
	
	\maketitle
	
	\begin{abstract}
		Random walks on expanders play a crucial role in Markov Chain Monte Carlo algorithms, derandomization, graph theory, and distributed computing. A desirable property is that they are rapidly mixing, which is equivalent to having a spectral gap $\gamma$ (asymptotically) bounded away from $0$. 
		
		Our work has two main strands. First, we establish a dichotomy for the robustness of mixing times on edge-weighted $d$-regular graphs (i.e., reversible Markov chains) subject to a Lipschitz condition, which bounds the ratio of adjacent weights by~$\beta \ge 1$.
		\begin{itemize}
			\itemsep=0pt
			\item 
			If $\beta \ge 1$ is sufficiently small, then $\gamma \asymp 1$ and the mixing time is logarithmic~in~$n$.
			
			\item 
			If $\beta \ge 2d$, there is an edge-weighting such that $\gamma$ is polynomially small in $1/n$.
		\end{itemize}

		Second, we apply our robustness result to a time-dependent version of the so-called $\varepsilon$-biased random walk, as introduced in Azar et al.~[Combinatorica 1996].
		\begin{itemize}
			\itemsep=0pt
			\item 
			We show that, for any  constant $\varepsilon>0$, a bias strategy can be chosen adaptively so that the $\varepsilon$-biased random walk covers any bounded-degree regular expander in $\Theta(n)$ expected time, improving the previous-best bound of $O(n \log \log n)$.
			
			\item 
			We prove the first non-trivial lower bound on the cover time of the $\varepsilon$-biased random walk, showing that,
			on bounded-degree regular expanders,
			it is $\omega(n)$ whenever $\eps = o(1)$.
			We establish this by controlling how much the probability of arbitrary events can be ``boosted'' by using a time-dependent bias strategy.
		\end{itemize}
	\end{abstract}
	
	\clearpage

	\section{Introduction }
	
	A (simple) random walk on a graph is a stochastic process, which starts from one specific vertex, and transitions at each iteration to a randomly chosen neighbour. This can be generalized to so-called reversible Markov chains~\cite{levin2009markov}, which are equivalent to random walks on undirected graphs with edge weights. Markov chains are related to many physical and statistical models, as well as to diffusion processes (e.g., chip-firing or load-balancing). Also, many optimization, sampling, counting and derandomization problems are only known to be tackled efficiently by setting up a suitable Markov chain and running this chain long enough.
	
	\paragraph{Mixing Time.}
	The \emph{mixing time} is one of the most fundamental quantities related to Markov chains, and captures the convergence speed towards the stationary, i.e., limiting distribution. Markov chains that converge in logarithmic time are called ``rapidly mixing'', and can also be characterized by a constant gap in the spectrum of the transition matrix. In turn, the spectral gap itself is closely related to a variety of other combinatorial and geometric properties, including the edge conductance, vertex expansion and canonical paths.
	
	Despite the significance of the mixing time (and spectral gap) of Markov chains, it seems only a comparably limited number of works have studied the impact of ``small'' perturbations or changes. 
	In general, it is straightforward to verify that if the edge weights of a graph are perturbed by at most a factor which is uniform over all edges (if the given graph has no edge weights, this is equivalent to assuming all edge weights are $1$), then the spectral gap will only change by at most the same factor.
	This can be verified directly using the Dirichlet characterisation, but one can also appeal to so-called ``Markov Chain Comparison Techniques''~\cite{DGJM06}.
	However, for the mixing time, Ding and Peres \cite{DP13} showed a negative result, demonstrating that for some chains, a constant perturbation of the edge weights can result into a super-constant change in the (total-variation) mixing time.
	Previous works on ``robustness'' of Markov chains convergence tended to be focussed on the \textit{total-variation mixing time}, rather than the spectral gap, and most importantly, impose a \emph{global} condition on the edge weights. 
	
	In this work, we investigate a less restrictive, Lipschitz condition on the edge weights, 
	where the ratio of two neighbouring edge weights must be bounded by $\beta \geq 1$. Local Lipschitz conditions are natural restrictions and occur in many areas such as Lipschitz continuous functions, quasi-isometries or concentration of measure.
	However, with this merely local condition, edge weight-discrepancies may accumulate along shortest paths, causing discrepancies between (non-neighbouring) edge weights which can be polynomial (in the number $n$ of vertices), even in bounded-degree expanders, and potentially exponential in less well-connected graphs. 
	This can also have similar effects on the stationary distribution. To the best of our knowledge this setting has not been studied before, and one aim of this work is the following.
	
	\begin{que} 
		Given a graph with sufficiently good expansion (e.g., an expander), what is the impact of the Lipschitz parameter $\beta \geq 1$ on the mixing time? 
	\end{que}

	Another work which could be considered as a global change to edge weights is by the first author and Zanetti~\cite{Luca&Sam}, where the fastest mixing time of any edge-weighting subject to the stationary distribution being uniform (or close to uniform) was considered. This problem was initially introduced by Boyd, Diaconis and Xiao~\cite{BDX04}. Further noteworthy variants of random walks include so-called non-backtracking random walks, meaning that each step the most recently used edge is avoided. For graphs with sufficiently strong expansion (i.e., Ramanujan graphs), Alon, Benjamini, Lubetzky and Sodin~\cite{AlonBack} showed that the mixing time of non-backtracking random walks can be up to a factor of two faster compared with the simple random walk.

	\paragraph{Cover Time.}
	The \emph{cover time} of a graph is another fundamental quantity, measuring the expected number of steps until all vertices are visited by the random walk. It arises naturally as the random walk offers a light-weight, low-memory approach to explore an unknown network. In fact, one of the first applications of the cover time of random walks was a polynomial-time algorithm for deciding connectivity \cite{UniTrans}. There are also deep connections between the cover time, the effective resistance of electrical networks, and Gaussian free fields~\cite{DLP12}. 
	
	One fundamental lower bound is that the cover time of the simple random walk on any $n$-vertex graph is at least $\Omega(n \log n)$~\cite{FeigeLower}; for some graphs, it may be as large as $\Theta(n^3)$~\cite{FeigeUpper}. This has led to a number of proposed modifications of the random walk with the goal of reducing the cover time. For example, Berenbrink, Cooper and Friedetzky~\cite{BCF15} studied ``greedy'' variants of random walks, which perform random walk transitions, but whenever available, take an unvisited edge. They analyzed different specific rules, and proved regular expanders of even-degree, which satisfy a certain technical condition (having logarithmic girth suffices), have cover time $O(n)$. For random $d$-regular graphs, where $d=3$ or $d\geq 4$ is even, Cooper, Frieze and Johansson \cite{CFJ18} proved tight bounds on the cover time of this greedy random walk variant. Cover times of similar greedy processes were also studied in \cite{BCERS10,OS14}.

	In this work we investigate the following way of augmenting a random walk, introduced by Azar, Broder and Karlin \cite{ABKLPbias}, and called the $\eps$-biased random walk ($\eps$-BRW) with parameter $\eps\in[0,1]$. In this process a controller has partial control over a random walk on a graph, and at each step with probability $\eps$ the controller can choose a neighbour of the current vertex, otherwise a uniformly random neighbour is selected. 
	Aside from the field of random walks, these processes also serve as a basic model for any randomized algorithm which relies on random bits to find a correct solution, in particular certain search problems~\cite{UniTrans,BirthdayParadoxPollard,Kangaroo,Pollard,SoSt}.
	
	Georgakopoulos, Haslegrave, and the second and third authors~\cite{ITCSpaper,POTC,ETB} extended the $\eps$-$\brw$ to allow for time-dependent strategies, and called this the $\eps$-time-biased random walk ($\eps$-$\tbrw$). The motivation for this was to investigate how 
	much the controller can reduce the cover time, as in this setting it is optimal to allow the strategy to evolve with time.
	The authors of~\cite{ITCSpaper,POTC} also introduced a related process called the \emph{choice random walk} (CRW) which is in the spirit of the power-of-two-choices \cite{MR1966907}. In this process, at each step \emph{two} randomly chosen neighbours are presented to a controller that needs to pick one of them. This process has since been used to sample spanning trees \cite{archer2024randomchoicespanningtrees}. While \cite{ITCSpaper,POTC,ETB} presented several bounds on the cover time of these processes, the best known bound for the CRW or $\varepsilon$-TBRW (with $\varepsilon <1$ constant) on expander graphs is $O(n \log \log n)$. Due to the alternative approaches mentioned earlier that lead to a linear cover time of $O(n)$ on \emph{some} expander graphs, and the fact that the only lower bound for the $\eps$-$\tbrw$ is $\Omega(n)$, the following open question is natural.
	
	\begin{que} Given a bounded-degree expander graph, is it possible to achieve a cover time of $O(n)$ using the CRW or the $\eps$-$\tbrw$ with a constant $\eps < 1$? 
	\end{que}
	
	Rather than minimizing cover time of expanders, one may also consider minimizing the worst-case cover time over all graphs. David and Feige~\cite{DF18} proved that the worst-case $\Theta(n^3)$  upper bound on the cover time can be reduced to $\Theta(n^2)$ by some local re-weighting scheme biasing the walk towards low-degree vertices, and that this is the best possible for reversible walks. The same process was studied by Cooper, Frieze and Petti~\cite{CFP18} on sparse Erd\H{o}s--R\'enyi random graphs. 
	
	\paragraph{Contributions.}
	We first address Question 1 by establishing a dichotomy in the Lipschitz parameter $\beta \ge 1$.
	Aided by this result, we answer Question 2 affirmatively for \emph{any} constant $\eps > 0$ in the regular case. Unlike previous studies on random walk variants achieving a linear cover time on expanders, our result does not impose any stronger requirement on the expansion, such as a high girth.
	The resolution of Question 2 suggests another question.
	\begin{que}
		Given a bounded-degree expander graph, is it possible to achieve a cover time of $O(n)$ using the $\eps$-$\tbrw$ with $\eps = o(1)$? 
	\end{que} 
	We resolve this question negatively by proving a lower bound which shows that the cover time of bounded-degree regular expanders is $\omega(n)$ whenever $\eps=o(1)$. One key tool to establish this is a general result which bounds how much the probability of any event of the SRW can be boosted by using the $\eps$-$\tbrw$.

	\subsection{Our Results}
	Our first two results concern the robustness of the spectral gap of a random walk on an edge-weighted graph subject to bounded local changes to the edge weights; the first is a positive result, and the second is negative. To ease the burden of notation these results are not stated in their strongest form and with slightly informal definitions; see Sections \ref{sec:prelim} and \ref{sec:robustness} for more detailed definitions and statements respectively. 
	
	All graphs $G=(V,E)$ in this work are undirected, connected and have no multiple edges, and $n=|V|$ always denotes the number of vertices. The \emph{simple random walk} (SRW) on $G$ refers to the (non-lazy) Markov chain which in each step moves to a uniformly and independently selected neighbour of the current vertex.  Given a graph $G=(V,E)$ and an edge-weighting $w : E \to \mathbb R_+$, we say that $w$ is $\beta$-Lipschitz if the ratio of weights between any neighbouring edges is at most $\beta \geq 1$ (cf.~\cref{def:lipschitz}). We define the reversible Markov chain $P_w$ to be the Markov chain with transition probabilities proportional to the edge weights $w$, and call this the chain \emph{induced by $w$}.  Finally, let $\gamma$ and $\gamma_{P_w}$ be the spectral gaps of the simple random walk on $G$ and $P_w$, respectively.
	A graph $G$ (more precisely, a sequence of graphs) is called an \emph{expander} if its spectral gap $\gamma$ is (asymptotically) bounded away from 0;
	see \cref{eq:expander_def} for a formal definition.
	
	\begin{theorem}[Simplified Version of \cref{res:robust}]
		\label{res:robustlite}
		Let $G$ be a $d$-regular graph, $\beta \leq 1+\gamma/32$, and $w : E(G) \to \mathbb R_+$ be a $\beta$-Lipschitz edge-weighting. Then, the chain $P_w$ induced by $w$ satisfies
		\[
		\gamma_{P_w}
		\ge
		10^{-8} \cdot d^{-16/\gamma}.
		\]
	\end{theorem}
	This shows that the spectral gap of bounded-degree regular expanders is robust to such Lipschitz perturbations of the edge weights. We complement this with the following negative result demonstrating that the spectral gap can be very poor if the discrepancy between neighbouring edge weights is large compared with the degree $d$.
	
	\begin{proposition} [Corollary of \cref{pro:lower}]\label{pro:lowerspec}
		Let $G$ be a $d$-regular graph with diameter $D \geq 4$.
		Then, for any $\beta > 1$, there exists a $\beta$-Lipschitz edge-weighting $w : E \to \mathbb R^+$ such that the chain $P_w$ induced by $w$ satisfies 
		\[
		\gamma_{P_w}
		\le
		\beta^2 (d/\beta)^{\lfloor D/2 \rfloor-1}. 
		\]
	\end{proposition}
	
	When taken together, \cref{res:robustlite} and \cref{pro:lowerspec} imply a strong dichotomy result for Lipschitz-perturbed walks on bounded-degree regular expanders, where the spectral gap transitions from a constant (i.e., asymptotically bounded away from 0) to a polynomial in $1/n$, as $\beta$ ranges from a value sufficiently close to $1$ up to $2d$; the boundedness of the degree implies that the diameter $D$ is $\Theta(\log n)$.
	
	Equipped with \cref{res:robustlite}, we are able to prove the following bound on the cover time of the $\eps$-time-biased random walk on regular graphs. In the following theorems we denote by $\tetb(G)$ the expected cover time of the graph $G$,
	maximized over the starting vertex
	and
	minimized over all (time-dependent) strategies for the $\eps$-$\tbrw$.
	
	\begin{theorem}[Corollary of \cref{thm:cover}]\label{thm:covercheap}
		Let $G$ be any $d$-regular graph and $n$ sufficiently large. Then for any $0 \leq \eps \leq \gamma/128$,  
		\[
		\tetb(G) \leq 10^{10} \cdot d^{100/\gamma}\cdot n \cdot \min\{\eps^{-1}, \; \log n \}.
		\]
	\end{theorem}
	
	By defining $0^{-1} = \infty$ and $\min\{\infty, x\} = x$ for all $x$, we can set $\eps = 0$ in \cref{thm:covercheap} to obtain a bound on the expected cover time of the (usual) simple random walk. Doing so recovers the classical bound of $O(n \log n)$ for the $\srw$ on bounded-degree expanders, see \cite{aldousfill} or otherwise.

	More interestingly, for \emph{any} constant $\epsilon > 0$, applying the previous theorem to bounded-degree regular expanders yields an improvement over the best-known bound of $O(n\log\log n )$~\cite{ETB}. 
	
	\begin{corollary}\label{cor:timebiasedcover}
		For any bounded-degree regular expander graph $G$, and any constant $\eps > 0$,  \[\tetb(G) = \Theta(n).\] 
	\end{corollary}

	Georgakopoulos et al.~\cite{ITCSpaper} introduced the choice random walk (CRW), which at each step is offered a choice between two independently and uniformly sampled neighbours as the next vertex. The authors proved a bound of $O(n\log\log n)$ on the cover time of bounded-degree expanders and asked whether the cover time is $\Theta(n)$. However, by \cite[Proposition~1]{ITCSpaper}, for any graph $G$ of maximum degree $d$, and for any $ \eps \leq 1/d$, the
	choice random walk can emulate the $\eps$-TBRW. Thus by emulating the $\eps$-TBRW with a sufficiently small constant $0<\eps \leq 1/d$ and applying \cref{cor:timebiasedcover}, we resolve this natural open problem as a corollary of \cref{thm:covercheap}.  
	
	\begin{corollary}\label{cor:choicecover}
		For any bounded-degree regular expander graph $G$, 
		\[
		\tcrw(G) = \Theta(n),
		\]
		where $\tcrw(G)$ is the expected cover time of the choice random walk (see \cite{POTC}).
	\end{corollary}

	The linear upper bound for the $\eps$-$\tbrw$ holds for \emph{any} constant $\epsilon > 0$. We can also establish the following lower bound, which states that the cover time of any $d$-regular graph by the $\eps$-$\tbrw$ is $\omega(n)$ if $\eps = o(1/\log^2 d)$; here, $d := d(n)$ is allowed to depend on $n$.
	
	\begin{restatable}{theorem}{covlower}\label{thm:coveringlower}
		For any constant $C \geq 1$ there exists a constant $c=c(C)>0$ such that, for any $d$-regular graph $G$  and  $\eps \leq c/\log^2 d$, we have
		\[ \tetb(G) \geq Cn.\]  
		In particular, if $\eps = o(1/\log^2 d)$, where $d$ may depend on $n$, then
		\[ \tetb(G) = \omega(n). \]
	\end{restatable}
	
	This is the first non-trivial lower bound established for the cover time of the $\eps$-$\tbrw$. \cref{thm:coveringlower} shows that the restriction on $\eps$ in our upper bound on the cover time, described in \cref{cor:timebiasedcover}, is best possible -- that is, to cover a bounded-degree regular expander in linear time it is necessary and sufficient to have a fixed constant bias $\eps>0$. 
	
	\newcommand{\traj}{S}
	
	To prove \cref{thm:coveringlower} we establish \cref{nonregboundcheap}, which bounds the probability of any finite-time event under the $\eps$-$\tbrw$ by a function of the corresponding probability under the SRW (simple random walk). We first need some notation.   
	For $u\in V(G)$, $t \geq 0$, and any set $\traj$ of $t$-step trajectories, we write $p_{u,\traj}$ for the probability that the trajectory of a $t$-step $\srw$ starting from $u$ is in $S$. Let $q_{u,\traj}(\eps)$ be the corresponding probability under the $\eps$-$\tbrw$; while this probability will generally depend on the particular strategy used, we will omit this dependence for the ease of notation.
	
	\begin{theorem}[Simplified Version of \cref{nonregbound}] \label{nonregboundcheap}
		Let $G$ be any $d$-regular graph, $u\in V$, $\traj$ be a set of trajectories of length $t>0$ starting from $u$, and $\eps \leq 1/d^{2\eta}$ for any $0< \eta \leq 1$. Then, for any strategy taken by the $\eps$-$\tbrw$,
		\[
		q_{u,\traj}(\eps)
		\le
		\exp\bigl( 4 t / d^{\eta} \bigr)
		\cdot
		(p_{u,\traj})^{\eta/(1 + \eta)}. 
		\] 
	\end{theorem}
	
	The main strength of this theorem lies in its generality, not just holding for any graph, but also any event that can be described using a set of trajectories. For example, such events include covering the graph within $t$ steps, visiting a given set of vertices within $t$ steps, or returning to the start vertex within $t$ steps.

	\subsection{Proof Outlines and Techniques}\label{sec:outline}

	In this section we outline the proofs of our main results and the techniques used and introduced.
	
	\paragraph{Robustness of the Spectral Gap, Lower Bound (\cref{res:robustlite}).}
	Let us first outline the proof of this positive result, which is a lower bound on the spectral gap of $P_{w}$ in terms of the spectral gap of $P$, corresponding to a simple random walk on a regular graph $G$ with unit~edge-weights.
	Even though our result technically also applies to graphs with (vanishingly) small expansion and/or unbounded degree, we restrict ourselves here to $G$ being a bounded-degree expander.
	
	First, we switch from the one-step Markov chain $P_{w}$ to an ``amplified'' Markov~chain with transition matrix $P_w^{2K}:=(P_{w})^{2K}$; that is, we combine $2K$ steps of $P_w$ into one step. Here, $K$ is (roughly) inversely proportional to the vertex expansion of $G$. The point of this step~is to boost the vertex expansion of this Markov chain to a sufficiently large constant $\alpha$---namely,~to~$\alpha=e^2$.
	
	In order to prove that the spectral gap of $P_w^{2K}$ remains constant under the Lipschitz condition with small constant $\beta \geq 1$, we will prove the following statement.
	\begin{equation}\label{eq:claim}
		\textit{Any set $S$ with $\pi(S) \in [0,1/2]$ has a constant edge conductance with respect to $P_w^{2K}$.}
	\end{equation}
	This then implies that the edge conductance of $P_w^{2K}$ is bounded below by a constant, and the corresponding result for the spectral gap of $P_w^{2K}$ follows by Cheeger's inequality. Finally, the spectral gaps of $P_{w}$ and $P_w^{2K}$ can be easily related, as $P_w^{2K}=(P_{w})^{2K}$.
	
	To show \eqref{eq:claim}, we first partition the vertices $S$ into buckets $(V_i)_{i \geq 1}$ with exponentially decreasing $\pi(\cdot)$ values; see \cref{def:Vipartition}.
	Since $\beta$ is bounded by some sufficiently small constant which is strictly greater than $1$, and by our choice of the partition $(V_i)_{i \geq 1}$, it follows that any vertex in $V_i$ can only be connected to vertices in $V_{i-1}, V_{i}$ or $V_{i+1}$; see \cref{lem:robust:lip-neighbours}.
	Also, due to the boosted expansion in $P_w^{2K}$, the set $S \cap V_{i}$ ``expands'' unless we have substantially more vertices in $S \cap V_{i+1}$, which in turn necessitates $\pi(S \cap V_{i+1}) > \pi(S \cap V_{i})$. However, as $\pi(S)$ is bounded, this cascading argument must stop eventually, which leaves us with 
	a $V_j$, for some $j \geq i$, where $\pi(S \cap V_j)$ is sufficiently large compared with $\pi(S \cap V_{j+1})$, which implies a sufficiently large expansion. This is formalized in \cref{res:robust:B2(S)-S}, and an illustration of the construction of these subsets is given in \cref{fig:illustration}. This result, combined with suitable bounds on the stationary mass, can then be translated to a lower bound on the ergodic flow; see \cref{res:robust:Q(Sl)>pi(Sl)}.
	By aggregating the contribution to the ergodic flow over suitable subsets $S \cap V_{j}$, we can then obtain a lower bound on the edge conductance of $P_w^{2K}$ for any subset $S \subseteq V$.

	\paragraph{Robustness of the Spectral Gap, Upper Bound (\cref{pro:lowerspec}).}
	
	The proof of this negative result, which says that the mixing time can be polynomial if $\beta \ge 2d$, is less involved and also applies to a larger class of graphs.
	In order to construct a $\beta$-Lipschitz edge-weighting with small spectral gap, we simply take two vertices $u,v \in V$ with $\dist(u,v)$ as large as possible, i.e., equal to the diameter $D$. Then, we apply a weighting scheme that assigns weight 
	\[
	w(x,y)
	:=
	\beta^{-\dist(\{x,y\},\{u,v\})}
	\quad \text{to any} \quad
	\{x,y\} \in E(G). 
	\]
	This weighting scheme creates a strong bias towards the two vertices $u$ and $v$, and thereby results into a bottleneck between the two. More formally, we can prove that once $\beta$ exceeds a suitable value depending on the degree $d$, the edge conductance of $P_w$ is at most polynomial in $1/n$. Hence, by Cheeger's inequality, also the spectral gap is at most polynomial in $1/n$, implying that the mixing time is at least polynomial in $n$.

	\paragraph{Upper Bound on Cover Times of the Time-Biased Walk~(\cref{thm:covercheap}).}
	
	Our strategy to cover all $n$ vertices in $\Theta(n)$ expected time, first divides the process into $\log_2 n$ consecutive phases, each of which halves the number of unvisited vertices. In the following, we focus on the analysis of one such phase. We fix $\varepsilon>0$ to be a sufficiently small constant, which we may assume since increasing $\varepsilon$ merely permits more bias. 
	
	At the beginning round $t$ of each phase, we fix the set of unvisited vertices $U=U(t)$. Then we consider a time-independent, biased random walk towards the set $U$, which we call $Q:=Q(U,\eps)$, defined by the following edge weights:
	\begin{equation*}
		\label{edgeweights_rep}
		w(u,v)
		:=
		(1-\eps)^{ \max\{ \dist(u,U), \,\dist(v,U) \}}
		\quad \text{for any} \quad
		\{u,v\} \in E(G).
	\end{equation*}
	This strategy ensures that the stationary distribution of each vertex $u\in U$ is ``boosted'' compared with $1/n$ (its value in the unweighted regular graph $G$), and the effect is larger the smaller $U$~is:
	\[
	\pi_{Q}(u)
	\geq
	\frac{1}{2d|U|} \cdot \left( \frac{|U|}{n} \right)^{1+\frac{\log(1-\eps)}{\log(d)}}
	=
	\frac1{2dn}
	\left( \frac{n}{|U|} \right)^{-\log(1-\eps)/\log d};
	\]
	see \cref{lem:stationary_boost}.
	In order to benefit from this boosted stationary probability, we first use some time steps in order to mix the biased random walk. Here, we rely on our earlier robustness result, which implies that if $\eps>0$ is a sufficiently small constant, then the spectral gap of this weighted random walk $Q:=Q(U,\eps)$ is still constant; this in turn implies an $O(\log n)$ mixing time. By alternating between consecutive segments of length $\sqrt{n}$ for mixing and covering vertices in $U$, we finally obtain the following upper bound on the expected number of rounds until the number of unvisited vertices is halved:
	\begin{align}
		\kappa\cdot \left( \frac{n}{|U|} \right)^{\log(1-\eps)/\log d} \cdot n, \label{eq:phase_time}
	\end{align}
	where $\kappa > 0$ is a (large) constant; see \cref{lem:keylemma_cover} for more details. 
	
	The proof is naturally concluded by aggregating the upper bounds on the expected time needed for each phase \cref{eq:phase_time}, which sum to $O(n)$.
	It is crucial that the $\eps$-$\tbrw$ is allowed to have a time-dependent bias, as in each new phase we need to update the set of unvisited vertices $U$.

	\paragraph{Lower Bound on ``Boosting'' Probabilities by the Time-Biased Walk~(\cref{nonregboundcheap}).}
	To prove this result we utilise the ``trajectory-tree'' method pioneered in \cite{ITCSpaper,POTC,ETB}, however we apply this in the opposite direction to previous works, presenting additional challenges. The trajectory-tree $\mathcal{T}$ encodes all possible trajectories of length $t$ in a graph $G$ started from a vertex $u$ using a tree $\mathcal{T}$ of height $t$. Each node in the tree $\mathcal{T}$ at distance $\ell$ from the root represents a unique trajectory of length $\ell$ in $G$ started from $u$. This enables us to consider any event $T$ which can be described using a set of trajectories of length $t$ from a given vertex $u$.  
	
	We now summarise the key steps of the trajectory-tree method at a high level.
	
	\begin{enumerate}
		\item 
		Given any set $\traj$ of trajectories of length $t$ starting from $u$, we can represent $\traj$ using the leaves of $\mathcal T$ by assigning each leaf weight $1$ if the corresponding trajectory is in $\traj$, and $0$ otherwise.
		Any time-dependent strategy for the $\eps$-time-biased random walk can then be represented on the trajectory tree as internal node weights, where the weight of a node $\mathbf{x}$ in the tree is interpreted as the probability of event $\traj$ holding for the $\eps$-$\tbrw$, conditional on the walk having thus far taken the trajectory represented by the internal node $\mathbf{x}$.
		
		\item 
		We consider a function from $[0,1]^d$ to $[0,1]$ which correctly encodes the probability of the event $\traj$, conditional on the current trajectory (node), as a function of the conditional probabilities of the $d$ trajectories which extend the current one (the children of the current node).
		In our case, we take the \textit{biased operator} defined for $\mathbf{v}=(v_1, \dots, v_d)\in [0,1]^d$~by 
		\begin{equation*}\ts 
			\operatorname{B}_{\eps,\mathbf{b}}\left(\mathbf{v}  \right) := \sum_{i=1}^d \left(\frac{1 - \eps }{d} + \eps b_i\right)  \cdot  v_i, 
		\end{equation*}  where $\mathbf{b}=(b_1, \dots, b_d)$ is a probability vector encoding the strategy of the walk. 
		
		\item\label{itm:keyineq} The key step is to find a way to bound the function above (in this instance, $\operatorname{B}_{\eps,\mathbf{b}}$) in terms of a power-mean of its input and some other factors. 
		More precisely, we show that
		\[
		\operatorname{B}_{\eps,\mathbf{b}}(\mathbf{v})
		\le
		e^{4d^{-\eta}} \cdot M_{(1+\eta)/\eta}(\mathbf{v})
		\Quad{for all}
		\mathbf v \in [0,1]^d,
		\] 
		where $\eps = d^{-2\eta}$, for any $0<\eta:=\eta(d)\leq 1$, and any probability vector $\mathbf{b}$;
		see \cref{lem:conv}.
		
		\item 
		We define an appropriate potential function $\Upsilon^{(i)}$, which has the following properties:
		\begin{enumerate}[(i)]
			\item\label{itm:one} 
			$\Upsilon^{(0)}$ is a function of the probability of event $T$ under the $\eps$-TBRW;
			
			\item 
			$\Upsilon^{(t)}$ is a function the probability of event $T$ under the simple random walk;
			
			\item\label{itm:drop} 
			$\Upsilon^{(i)}\leq \Upsilon^{(i+1)}$ for all $0 \leq i \leq t-1$.
		\end{enumerate}
		If the potential function has been designed correctly (as in our case) then \eqref{itm:drop} will follow from the inequality shown in step \eqref{itm:keyineq}.
		From properties \eqref{itm:one}-\eqref{itm:drop}, we have $\Upsilon^{(0)} \leq \Upsilon^{(t)} $, and thus have bound on the probability of $\traj$ under the $\eps$-TBRW by a function of the corresponding probability under a simple random walk, establishing the result.
	\end{enumerate}
	
	As mentioned above, this approach was previously used to prove lower bounds, where a specific (greedy) strategy can be fixed. In our case, we have to argue over all strategies. This, combined with the reverse in direction (effecting Step \eqref{itm:keyineq}), leads to a more intricate potential function and power-mean inequality compared with previous applications of the method.  
	
	\paragraph{Lower Bound on the Cover Time of the Time-Biased Walk~(\cref{thm:coveringlower}).}
	
	\cref{thm:coveringlower} is proved using the full version of \cref{nonregboundcheap} (i.e., \cref{nonregbound}) in combination with Dubroff and Kahn~\cite[Theorem~1.1]{dubroff2021linear} (stated as \cref{thm:dubkahn}) which shows that the $\srw$ is exponentially unlikely to cover any graph in linear time. 
	
	To begin we let $\traj$ to be the event that $G$ is covered by a walk of length $t = 3Cn$, where $C>0$ is a constant, from an arbitrary vertex $u$.
	\cref{thm:dubkahn} then gives an $\alpha := \alpha(C) > 0$ such that
	\begin{equation}\label{eq:srwprob}
		p_{u,\traj}
		\leq
		e^{-\alpha t}.
	\end{equation}
	The next step is to use \cref{nonregbound} to bound the probability $q_{u,\traj}$ of the same event under the $\eps$-$\tbrw$ as a function of $p_{u,\traj}$ (the SRW probability).
	With the right choice of parameters this probability is less than $1/2$. Thus, the expected time to cover the graph from $v$ is at least \[3Cn\cdot 1/2 >Cn .\]  
	
	When bounding $q_{u,\traj}$, there are two cases here depending on the size of the degree $d$. When the degree $d$ is at most some constant $\Delta:=\Delta(C)$, given explicitly in the proof, we use the cruder bound on $q_{u,\traj}$ from \cref{nonregbound}.
	The more interesting case is $d > \Delta$, where we set
	\[
	\eta
	:=
	\frac{\log\log d}{\log d}
	\in
	(0,1],
	\]
	which, combined with the way we choose $\Delta$ and the bound \eqref{eq:srwprob}, yields the required bound of
	\[
	q_{u,\traj}
	\le
	\exp\bigl(4td^{-\eta}\bigr) \cdot p_{u,\traj}^{\eta/(1+\eta)}
	\leq
	1/2.
	\]
	
	\subsection{Organization}
	We begin in \cref{sec:prelim} with some notation and definitions. In \cref{sec:robustness} we prove our results relating the spectral gap (and edge conductance) to the Lipschitz constant. These results are then applied in \cref{sec:cover} where we prove our upper bound on the cover time of the bias random walk. In \cref{Boost} we then prove our upper bounds on the ability of the bias walk to `boost' probabilities of events, and apply this to prove our lower bound on the cover time of the bias random walk. We conclude with some open problems in \cref{sec:conc}. 
	
	\section{Preliminaries}\label{sec:prelim}
	
	\paragraph{General Notation.}
	We take $\mathbb{N}:=\{0,1,\dots\}$ to denote the natural numbers including $0$, and $\mathbb{R}_+$ to be the positive reals. Throughout we take $\log x $ to be base $e$, i.e.~we take $\log x :=\ln x$ to be the natural logarithm. 
	
	We follow standard graph notation and use $x\sim y$ for the edge relation in an undirected graph. Given a graph $G=(V,E)$ and a subset $S\subseteq V$, we let $\Gamma(S):=\{u  : \{u,s\} \in E \text{ for some }s \in S\}$ be the neighbourhood of $S \subseteq V$ in $G$.   In the case of a singleton set $S=\{v\}$ we denote  $\Gamma(S)$ by $\Gamma(v)$  and call $d(v):=|\Gamma(v)|$ the degree of $v$.
	For two sets $A,B\subseteq V$, we define
	\[E(A,B):=\{\{a,b\} \in E(G) : a\in A, b\in B \}\]
	as the set of edges crossing between $A$ and $B$ and define $e(A,B):=|E(A,B)|$. Throughout,  all graphs are undirected, connected and have no multiple edges. We use $n=|V|$ to denote the number of vertices. By $\dist(u,v)$ we denote the standard graph distance between two vertices $u$ and $v$, and let $D:= \max_{u,v\in V}\dist(u,v)$ denote the diameter of $G$. We also extend this notation to any two sets $A,B \subseteq V$, where we let $\dist(A,B):=\min_{a \in A, b \in B} \dist(a,b)$. For any non-empty subset $U \subseteq V$ and integer $k \geq 0$ we define $\mathcal{B}_{k}(U) := \{ v  \in V \colon \dist(v,U) \leq k \}$.
	We use $\dmax$ and $\dmin$  to denote the maximum and minimum degrees of a graph respectively.

	The \emph{vertex expansion} $\Psi_G$ of a graph $G$ is defined as
	\begin{align}
		\Psi_G
		\cq
		\min_{S \subseteq V : 0 < \abs S / \abs V \le 1/2}
		\frac{\abs{ \Gamma(S) \setminus S }}{\abs S}
		\in
		(0,2]. \label{eq:vertex_exp_def}
	\end{align}
	A sequence of regular graphs $(G_i = (V_i,E_i))_{i \geq 1}$ is called an \emph{expander} if
	\begin{align}\label{eq:expander_def}
		\text{$(i)$ there is a constant $\psi > 0$ such that $\Psi(G_i) \ge \psi$ for all $i \ge 1$}
		\Quad{and}
		\text{$(ii)$ $\lim_{i \to \infty} |V_i|=\infty.$}
	\end{align} 
	As is standard in the area, for brevity we will frequently speak of an expander graph (or just expander) in the singular sense; this should be interpreted as a sequence of graphs.

	\paragraph{Markov Chains.}
	Given a Markov chain $P$ with transition probabilities $P(x,y)$, let $P^t(x,y)$  denote the probability the walk started at state $x$ is at $y$ after $t$ steps.
	The trajectory of the Markov chain is denoted by $(X_t)_{t \in \mathbb{N}}$, where each $X_t$ takes some value in the state space, which in our case will always be $V=V(G)$.
	All Markov chains considered in this work are finite and irreducible, which means there is a unique stationary distribution $\pi_{P}$.
	If no specific chain $P$ is given, then we assume that $\pi=\pi_{P}$ where $P$ is the transition matrix of a simple random walk (SRW), i.e.\ $P(x,y) = 1/d(x)$ if $\{x,y\} \in E$ and $0$ otherwise.
	
	Given a reversible (and irreducible) Markov chain $  P$, let $1 = \lambda_1 > \lambda_2 \ge \cdots \ge \lambda_n \ge -1$ denote the eigenvalues of $P$.
	The \textit{spectral gap} is defined as
	\[
	\gamma_{P} := 1-\lambda_2.
	\]
	The (total-variation) \emph{mixing time} of an irreducible, aperiodic Markov chain is defined as
	\[
	\tmix(P) := \max_{u \in V} \min\left\{ t \geq 1  \colon \left\| P^t(u,\cdot) - \pi \right\|_{\operatorname{TV}} \leq 1/4 \right\},
	\]
	where for two probability vectors $p,q$, $\| p - q \|_{\operatorname{TV}} := \frac{1}{2} \sum_{w \in V} \left| p(w) - q(w) \right|$. For Markov chains where $|\lambda_n| \geq \lambda_2$, one often switches to the ``lazy'' version of $P$, $\widetilde{P}:=(P+I)/2$, where $I$ denotes the identity matrix.
	For such a ``lazy'' Markov chain, it is well-known that $\tmix(\widetilde{P})$  is proportional to $1/\gamma_{\widetilde P}$, up to a factor of $\log(\min_{u \in V} 1/\pi(u))$; see, e.g.,~\cite{levin2009markov}.
	
	Next, the \emph{edge conductance} of $P$ is defined as 
	\[
	\Phi_P
	\cq
	\min_{S \subseteq V : 0 < \pi(S) \le 1/2}
	\frac{Q_{P}(S, S^c)}{\pi(S)}
	\Qwhere
	Q_{P}(S, S^c)
	\cq
	\sum_{x \in S, y \notin S}
	\pi(x) P(x,y)
	\Qfor
	S \subseteq V.
	\] 
	Recall the Cheeger inequality \cite[Lemma~3.3]{JS:approx-counting}: for any reversible Markov chain $P$,
	\begin{equation}\label{eq:cheeger}
		\Phi_{P}^2/2 \leq \gamma_P \leq 2\Phi_P.  
	\end{equation} 
	For $d$-regular graphs and $P$ being the SRW, the edge conductance can be rewritten as
	\[
	\Phi_P
	=
	\min_{S \subseteq V : 0 < \pi(S) \leq 1/2} \frac{e(S,S^c)}{d |S|},
	\Quad{where} e(S,S^c):=|E(S,S^c)|
	\]
	is the number of edges crossing between $S$ and $S^c=V \setminus S$.
	For $d$-regular graphs,
	we have
	\begin{equation}
		\label{eq:vertexedgeconduct}
		\Psi_G / d
		\le
		\Phi_P
		\le
		\Psi_G.
	\end{equation}
	Indeed, each vertex outside but adjacent to $S$ connects to at least $1$ and at most $d$ inside $S$:
	\[
	\abs{\Gamma(S) \setminus S}
	\le
	e(S, S^c)
	\le
	d \abs{\Gamma(S) \setminus S}
	\Quad{for all}
	S \subseteq V. 
	\] 
	Combining this with \eqref{eq:cheeger}, we can also relate the vertex expansion and spectral gap. Particularly,
	\begin{align}\label{eq:vertexspectral}
		\Psi_G \ge \Phi_P \ge \tfrac12 \gamma_P.
	\end{align}
	These relations between the three expansion measures also give rise to equivalent definitions of expanders, simply replacing the vertex expansion in \cref{eq:expander_def} by the spectral gap or edge conductance.

	\paragraph{Biased Random Walks.}
	Building on earlier work \cite{ben1987collective}, Azar et al.\ \cite{ABKLPbias}  introduced the $\eps$-biased random walk ($\eps$-BRW) on a graph $G$, where $\eps\in[0,1]$ is a parameter of the model. Each step of the $\eps$-BRW is preceded by an $(\eps, 1 - \eps)$-coin flip. With probability $1 -\eps$ a step of the simple random walk is performed, but with probability $\eps$ the controller gets to select which neighbour to move to. The selection can be probabilistic, but it is time independent. Thus, if $P$ is the transition matrix of the simple random walk, then the transition matrix $  Q^{\eps\text{B}}$ of the $\eps$-biased random walk is given by \begin{equation}\label{bias} Q^{\eps\text{B}} = (1 - \eps)\cdot  P + \eps\cdot  B,\end{equation}
	where $  B$ is an arbitrary (but fixed) stochastic matrix chosen by the controller, with
	support restricted to $E(G)$. The controller of an $\eps$-BRW has full knowledge of $G$.
	
	Azar et al.\ \cite{ABKLPbias} focused exclusively on fixed bias matrices that do \emph{not} depend on time. The definition of the  $\eps$-BRW was relaxed in~\cite{ITCSpaper} where the bias matrix $ {B}$ was allowed to depend on the full history of the walk. 
	The motivation was that to cover a graph (visit every vertex) efficiently, one must include strategies which depend on time. 
	For example consider the $\eps$-$\tbrw$ started in the middle of an $n$-vertex path, then we can bias towards one end of the path, and once this is reached we bias towards the other for a total cover time of $\Theta(n/\eps)$.
	However, it is shown in \cite[Theorem~1]{IkedaKY09} that the cover time of \textit{any} fixed transition matrix supported on a path is $\Theta(n^2)$, so unlike the $\eps$-$\tbrw$, the $\eps$-$\brw$ does not significantly improve over the simple random walk. 
	
	Let $\mathcal{F}_t$ be the history of the random walk up to time $t$, i.e., the $\sigma$-algebra $\mathcal{F}_t= \sigma\left(X_0, \dots,X_t \right)$ generated by all steps of the walk up to and including time $t$. The $\eps$-time-biased walk ($\eps$-TBRW) refers to a time-dependent random walk specified in the same ways as \eqref{bias}, however now the strategy $B$ of the controller is now encoded using the bias matrix $  B_t$, which may depend on $\mathcal{F}_t$.
	Let $\ETBcov vG$ denote the expected time for the $\eps$-TBRW to visit every vertex of $G$, when starting from $v$, minimized over all strategies. Define the \textit{cover time} $\tetb(G) := \max_{v \in V} \ETBcov vG$.
	
	Let us make this more precise. Our interest is in minimizing the cover time of a given graph $G$, which is done by applying the ``optimal'' time-dependent and adaptive strategy.
	As such, we can restrict ourselves to a set $\mathcal{B}:=\mathcal{B}(G)$ of strategies that specify, for any given subset of visited vertices $S \subseteq V$ and position of the random walk $u \in V$, a particular distribution $b(u,S)$ over the neighbours of $u$. These distributions $\{b(u,S)\}_{u\in V}$ correspond to the rows of a bias matrix $B(S)$,  so any strategy $B\in \mathcal{B}$ is given by a set of these matrices $\{B(S)\}_{S\subseteq V}$.
	
	The restriction to such bias matrices allows us to compute the cover time via lifting to a larger state space. In particular, if we let $\tilde{V}$ be the set of all pairs $(u,S)$ where $u \in V$ and $S\subseteq V$, then each  strategy $B\in \mathcal{B}$ corresponds to a fixed Markov chain $P_B$ with state space $\tilde{V}$. Furthermore, there exists a fixed set $\tilde{U}\subseteq \tilde{V}$ such that, for any $B\in \mathcal{B}$, the cover time under $B$ corresponds to the hitting time of the set $\tilde{U}$ by the Markov chain $P_B$ (see \cite[Lemma 7.5]{POTC} and \cite[Lemma 5.6]{ETB} for details). Thus, for a more rigorous definition of $\tetb(G)$, we can take $\ETBcov vG$ to be the minimum over all Markov chains $P_B$, where $B\in \mathcal{B}$, of the expected hitting time of the set $\tilde{U}$ by the chain $P_B$ started from  $(v,\{v\})\in \tilde{V}$.

	\paragraph{Weighted Walks.}
	Let $G=(V,E)$ be a graph and $w : E\rightarrow \mathbb{R}_+$ be a weighting of the edges of $G$. 
	We define the transition matrix $P_w$ on $V$ \textit{induced by} $w$ by
	\begin{equation}\ts
		\label{eq:weighted}
		P_w(x,y)
		=
		\begin{cases}
			w(x,y) / w(x)
			&\text{if}\quad
			x \ne y,
			\\
			0
			&\text{if}\quad
			x = y,
		\end{cases}
	\end{equation}
	where
	\(
	w(x)
	\cq
	\sum_{z: z\sim x}
	w(x,z)
	\)
	for all $x \in V$. We refer to the chain as the \textit{walk induced by $w$}.
	We denote by $\piw(\cdot):=\pi_{P_w}(\cdot)$ the stationary distribution with respect to $P_w$, given by
	\begin{equation}\label{eq:weights}
		\ts
		\piw(x)
		=
		w(x) / W
		\Qforall
		x \in V
		\Qwhere
		W
		\cq
		2
		\sum_{e \in E}
		w(e)
		=
		\sum_{x \in V}
		w(x).
	\end{equation}
	It is well known that the induced walk is reversible, and moreover that any reversible chain on $V$ with permitted transitions given by $E$ can be represented by a collection $w$ of edge weights. Finally, also for weighted walks, we may switch to the ``lazy'' version of $P$ by defining the transition matrix $\tilde{P}_w:=(P_w+I)/2$.
	
	The next definition is central to this work.
	\begin{definition}[Lipschitz Condition]\label{def:lipschitz}
		For a graph $G=(V,E)$, let $w : E \to \mathbb R_+$ be an edge-weighting of $G$. For a real number $\beta\geq 1$, we say that $w$ is \emph{$\beta$-Lipschitz} if
		\begin{equation}
			\label{eq:lipschitz}
			1 / \beta
			\le
			w(x,y) / w(x,z)
			\le
			\beta
			\Qforall
			x,y,z \in V
			\Qwith
			y \sim x \sim z.
		\end{equation}
	\end{definition}
	
	We now make a basic but crucial observation on the ratio of stationary probabilities of nearby vertices.
	
	\begin{lemma}[{cf.~\cite[Proposition~4.8]{ETB}}]\label{lem:pibound}Let $P_w$ be the transition matrix induced by a $\beta$-Lipschitz edge-weighting of any graph $G=(V,E)$. Then, for any $k\geq 0$, 
		\[
		\left(\frac{\dmin}{\dmax\cdot \beta^{2}}\right)^k
		\le
		\frac{\piw(x)}{\piw(y)}
		\le
		\left(\frac{\dmax\cdot \beta^{2}}{\dmin}\right)^k \Qforall x,y \in V \Qwith
		\dist(x,y) \le k.
		\]
	\end{lemma}
	\begin{Proof}
		Suppose that $x,y \in V$ with $\{x,y\} \in E$. Then,  
		\[
		\frac{\piw(x)}{\piw(y)}
		=
		\frac{\sum_{z : z \sim x} w(x,z)}{\sum_{z : z \sim y} w(y,z)}
		\le
		\frac{ \dmax  \cdot \beta w(x,y) }{ \dmin\cdot  w(y,x) / \beta }
		=
		\frac{\dmax}{\dmin}\cdot \beta^2.
		\]
		Iterating this, if $\dist(x,y) \le k$, then $\piw(x)/\piw(y) \le (\beta^2 \dmax/\dmin)^k $.
	\end{Proof}

	\section{Robustness of the Spectral Gap}\label{sec:robustness}
	
	In this section we explore the effect of the Lipschitz constant on the spectral gap and the edge conductance. 
	\subsection{Lower Bound on the Spectral Gap}
	\newcommand{\divline}{\bigskip\hrule\hrule\hrule\bigskip}
	
	The aim of this section is to prove the following theorem.
	It will be established via a sequence of lemmas.
	
	\begin{theorem}
		\label{res:robust}
		Let $G = (V,E)$ be a $d$-regular, connected graph with vertex expansion $\Psi_{G} \geq \psi$.
		Set $K \cq \ceil{2 / \log(1 + \psi)}$ and $\sigma \cq e^{1/(2K)}$.
		Then,
		for any $\sigma$-Lipschitz collection $w : E \to \mathbb R_+$ of weights, the weighted random walk $P_w$, given by \eqref{eq:weighted}, satisfies
		\[
		\Phi_{P_w^{2K}}
		\ge
		\tfrac1{4000}\cdot 
		d^{-2K},
		\quad\text{and hence}\quad
		\gamma_{P_w}
		\ge
		10^{-8} \cdot 
		d^{-4K}.
		\]
		
	\end{theorem}

	Throughout the remainder of this section, we fix a $\sigma$-Lipschitz weighting $w : E \to \mathbb R_+$, with $\sigma = e^{1/(2K)}$ and $K = \ceil{2/\log(1 + \psi)}$. Recall that $P_w^{2K}$ is the $(2K)$-step transition matrix $P_w^{2K}=(P_{w})^{2K}$ corresponding to the Markov chain where $2K$ consecutive transitions of $P_w$ are combined into one step. Observe that the stationary distribution of $P_w^{2K}$ is the same as $P_w$, and therefore, by \eqref{eq:weights}, $ \pi_{P_w^{2K}}(x)=\pi_{P_w}(x) = w(x) / W$, which we refer to as $\piw(\cdot)$ for brevity.

	For any vertex $u \in V$ and radius $k \in \mathbb{N}$, recall that
	$\mcb_k(u) = \bigl\{ v \in V \colon \dist(u,v) \le k \bigr\}$
	denotes the ball of radius $k$ around~$u$.
	We start with a graph of vertex expansion at least $\psi$, implying that balls grow at rate $1+\psi$:
	\[
	\abs{ \mcb_1(S) }
	=
	\abs{ S \cup ( \Gamma(S) \setminus S ) }
	=
	\abs S +
	\abs{ \Gamma(S) \setminus S }
	\ge
	(1 + \psi)
	\abs S
	\Quad{whenever}
	\abs S \le \tfrac12 n.
	\] 
	The idea is to boost the vertex expansion such that balls grow at a rate at least $\vex \cq e^2$ by taking the above inequality and iterating it $k$-times, which yields into a lower bound on the number of vertices with distance at most $k$. This will be later used to analyze the edge conductance~$\Phi_{P_w^{2K}}$. 
	
	\begin{lemma}
		\label{res:robust:ball-expansion}
		For all $S \subseteq V$ and any integer $k \geq 0$, we have
		\[
		\abs{ \mcb_k(S) }
		\ge
		\min\bigl\{ (1 + \psi)^k \abs S, \: \tfrac12 n \bigr\}. \]
		Consequently, $\abs{ \mcb_k(S) }\ge \min\bigl\{ \vex \abs S, \: \tfrac12 n \bigr\}$ for any $k\geq K$.
	\end{lemma}
	\begin{Proof}
		For any integer $i \in [1, k]$, the $\psi$-vertex expansion of $G$ implies that
		\[
		\abs{ \mcb_i(S) }
		=
		\abs{\mcb_{i-1}(S)\cup (\Gamma( \mcb_{i-1}(S)) \setminus \mcb_{i-1}(S))  }
		\ge
		(1 + \psi) \abs{ \mcb_{i-1}(S) }
		\Quad{whenever}
		\abs{\mcb_{i-1}(S)} \le \tfrac12 n.
		\]
		Note that $i \mapsto \abs{ \mcb_i(S) } : \mbn \to [0, n]$ is non-decreasing.
		Iterating the above gives the claim.
	\end{Proof}
	
	While the definition of the vertex expansion $\Psi_G$ (see~\cref{eq:vertex_exp_def}) only concerns estimates on the neighbourhood of sets of size below $\tfrac 12 n$, the next lemma provides such estimates for \emph{any} size.

	\begin{lemma}
		\label{res:robust:vertex-expansion}
		Let $G = (V, E)$ be any graph satisfying $\Psi_G \ge \psi > 0$.
		Then, for any $S \subseteq V$,
		\[
		\abs{ \Gamma(S) \setminus S }
		\ge
		\tfrac13 \psi \cdot
		\min\{ \abs{S}, \: \abs{S^c} \}.
		\]
	\end{lemma}
	
	\begin{Proof}
		We partition the vertex set $V$ into three disjoint parts: 
		\[
		S,
		\quad
		T \cq \Gamma(S) \setminus S
		\Qand
		R \cq V \setminus \bigl( \Gamma(S) \cup S \bigr).
		\]
		Then,
		\begin{align}
			\Gamma(S) \setminus S
			=
			T
			\supseteq
			\Gamma(R) \setminus R. \label{eq:partition}
		\end{align}
		
		Now, if $\abs S \le \tfrac12 n$, then $\abs{S^c} \ge \abs S$ and $\Psi_G \ge \psi$ together imply that
		\[
		\abs T
		=
		\abs{ \Gamma(S) \setminus S }
		\ge
		\psi \abs S
		\ge
		\tfrac13 \psi
		\min\{ \abs S, \abs{S^c} \}.
		\]
		It remains to consider the case $\abs S \ge \tfrac12 n$, for which $\abs{S^c} \le \abs S$ and $\abs R \le \tfrac12 n$.
		Using \cref{eq:partition},
		\[
		\abs T
		\ge
		\abs{ \Gamma(R) \setminus R }
		\ge
		\psi \abs R
		=
		\psi (\abs V - \abs S - \abs T).
		\]
		Rearranging,
		\[
		\abs T
		\ge
		\tfrac{\psi}{1+\psi}
		(\abs V - \abs S)
		=
		\tfrac{\psi}{1+\psi}
		\abs{S^c}.
		\]
		Finally, $\psi \le \Psi_G \le 2$, by considering an arbitrary set of size $\floor{\tfrac12 n}$, and so $1/(1+\psi) \ge \tfrac13$.
	\end{Proof}
	
	Ultimately, we wish to control the edge expansion of $P_w^{2K}$.
	We will do this by counting the number of neighbours at distance at most $K$ from a set $S$, i.e., by lower bounding $\abs{ \mcb_K(S) \setminus S }$.
	First, we partition $V$ into blocks whose vertices have exponentially decreasing $\piw$-values.
	
	\begin{definition}\label{def:Vipartition}
		For $i \in \mbpn$  
		and $S \subseteq V$, write
		\[
		V_i
		\cq
		\{ v \in V \mid \piw(v) \in (1/e^i, 1/e^{i-1}] \}
		\Qand
		S_i
		\cq
		S \cap V_i;
		\Quad{then,}
		S
		=
		\cup_{i \in \mbn}
		S_i.
		\]
	\end{definition}
	
	The key is that if $x \in V_i$ and $y \in V$ satisfies $\dist(x,y) \le K$, then $y \in V_{i-1} \cup V_i \cup V_{i+1}$.
	This follows directly from \cref{lem:pibound} and the definitions of $\sigma$ and $K$ in \cref{res:robust}.
	\begin{corollary}[of \cref{lem:pibound}]
		\label{lem:robust:lip-neighbours}
		Under the hypotheses of \cref{res:robust}, the following Lipschitz condition holds for $\piw$:
		\[
		1 / e
		= 
		1 / \sigma^{2K}
		\le
		\piw(x) / \piw(y)
		\le
		\sigma^{2K}
		=
		e
		\Qforall
		x,y \in V
		\Qwith
		\dist(x,y) \le K.
		\]
		In particular, this implies that, for all $i \in \mbn$, defining $V_{0} \cq \varnothing$,
		\[
		\mcb_K(V_i)
		\subseteq
		V_{i-1} \cup V_i \cup V_{i+1}.
		\]
	\end{corollary}

	We now find a ``representative'' collection $\mathcal I \subseteq \mbn$ of indices for
	\(
	S
	=
	\cup_{i \in \mbn}
	S_i
	\)
	and use the inclusion
	\[
	\mcb_K( \good ) \setminus S
	\subseteq
	\mcb_K(S) \setminus S
	\Qwhere
	\good
	\cq
	\cup_{i \in \mathcal I}
	S_i
	\subseteq
	\cup_{i \in \mbn}
	S_i
	=
	S.
	\]
	An illustration of this construction is given in \cref{fig:illustration}.
	
	\begin{definition}[``representative'' Indices]\label{def:good_indices}
		Fix an arbitrary set $S \subseteq V$ with $\abs S \le \tfrac12 n$.
		Recall that $\vex = e^2$.
		We define a finite sequence of integers $1 \le a_1 \le b_1 < a_2 \le b_2 < ...$ recursively.
		\begin{itemize}
			\item 
			$a_1$ is the smallest integer such that $S_{a_1} \ne \varnothing$.
			
			\item 
			Given $(a_1, b_1, a_2, b_2, \ldots, b_{\ell-1}, a_\ell)$,
			let $b_\ell$ be the smallest integer satisfying $b_\ell \ge a_\ell$ and
			\[
			\abs{ S_{b_\ell+1} }
			\le
			\tfrac12 \vex \abs{ S_{a_\ell} \cup S_{a_\ell+1} \cups S_{b_\ell-1} \cup S_{b_\ell} }. 
			\]
			
			\item 
			Given $(a_1, b_1, a_2, b_2, \ldots, a_\ell, b_\ell)$,
			let $a_{\ell+1}$ be the smallest integer satisfying $a_{\ell+1} > b_\ell$ and
			\[
			\abs{ S_{a_{\ell+1}} }
			>
			\abs{ S_{a_{\ell+1}-1} }.
			\]
		\end{itemize}
		This sequence terminates at some $b_L$: i.e., $(a_1, b_1, \ldots, a_L, b_L)$ are defined, however there is no $a_{L+1}$ satisfying the required conditions;
		for convenience, we set $a_{L+1} \cq \infty$.
		Set
		\[\ts
		\good_\ell
		\cq
		S_{a_\ell} \cup S_{a_\ell+1} \cup \cdots \cup S_{b_{\ell}}
		\Qfor
		\ell \in [L]
		\Qand
		\good
		\cq
		\cup_{\ell \in [L]}
		\good_\ell.
		\]
	\end{definition}
	
	\begin{figure}
		\begin{center}
	\scalebox{0.7}{\begin{tikzpicture}
			[scale=0.9,rec/.style={rounded corners=6pt},
			rrec/.style={rounded corners=6pt,thick,fill=red!20!,draw=red},
			grec/.style={rounded corners=6pt,thick,fill=green!30!,draw=green},
			nrec/.style={rounded corners=6pt,opacity=0.5,thick,dashed,fill=white,draw=black},
			brec/.style={rounded corners=6pt,opacity=0.5,thick,fill=blue!20!,draw=blue}
			]
			
			\draw[rec] (0,5) rectangle (18,6);
			\node () at (20,5.5) {$\piw(\cdot) \in (e^{-1},1]$};
			\node () at (8,6.25) {$V_1$};
			\node[anchor=west] () at (0,6.35) {$\textcolor{red}{|S_1|=0}$};
			\node () at (-1,5.5) {$1$};
			
			\draw[rec] (0,3) rectangle (18,4);
			\draw[grec] (0,3) rectangle (0.5,4);
			\node () at (20,3.5) {$\piw(\cdot) \in (e^{-2},e^{-1}]$};
			\node () at (8,4.25) {$V_2$};
			\node[anchor=west] () at (0,4.35) {$\textcolor{green}{|S_2|=2}$};
			\node () at (-1,3.5) {$2$};
			
			\draw[rec] (0,1) rectangle (18,2);
			\draw[grec] (0,1) rectangle (3,2);
			\node () at (20,1.5) {$\piw(\cdot) \in (e^{-3},e^{-2}]$};
			\node () at (8,2.25) {$V_3$};
			\node[anchor=west] () at (0,2.35) {$\textcolor{green}{|S_3|=12 \geq \frac{1}{2} \alpha \cdot 2}$};
			\node () at (-1,1.5) {$3$};
			
			\draw[rec] (0,-1) rectangle (18,0);
			\draw[grec] (0,-1) rectangle (13,0);
			\node () at (20,-0.5) {$\piw(\cdot) \in (e^{-4},e^{-3}]$};
			\node () at (8,0.25) {$V_3$};
			\node[anchor=west] () at (0,0.35) {$\textcolor{green}{|S_4|=52 \geq \frac{1}{2} \alpha \cdot (2+12)}$};
			\node () at (-1,-0.5) {$4$};

			\draw[rec] (0,-3) rectangle (18,-2);
			\draw[rrec] (0,-3) rectangle (17.5,-2);
			\node () at (20,-2.5) {$\piw(\cdot) \in (e^{-5},e^{-4}]$};
			\node () at (8,-1.75) {$V_4$};
			\node[anchor=west] () at (0,-1.65) {$\textcolor{red}{|S_5|=70 < \frac{1}{2} \alpha \cdot (2+12+52)}$};
			\node () at (-1,-2.5) {$5$};

			\draw[rec] (0,-5) rectangle (18,-4);
			\draw[rrec] (0,-5) rectangle (2,-4);
			\node () at (20,-4.5) {$\piw(\cdot) \in (e^{-6},e^{-5}]$};
			\node () at (8,-3.75) {$V_6$};
			\node[anchor=west] () at (0,-3.65) {$\textcolor{red}{|S_6|=8 < 70}$};
			\node () at (-1,-4.5) {$6$};
			
			\draw[rec] (0,-7) rectangle (18,-6);
			\draw[rrec] (0,-7) rectangle (1.25,-6);
			\node () at (20,-6.5) {$\piw(\cdot) \in (e^{-7},e^{-6}]$};
			\node () at (8,-5.75) {$V_7$};
			\node[anchor=west] () at (0,-5.65) {$\textcolor{red}{|S_7|=5 < 8}$};
			\node () at (-1,-6.5) {$7$};
			
			\draw[rec] (0,-9) rectangle (18,-8);
			\draw[grec] (0,-9) rectangle (1.5,-8);
			\node () at (20,-8.5) {$\piw(\cdot) \in (e^{-8},e^{-7}]$};
			\node () at (8,-7.75) {$V_8$};
			\node[anchor=west] () at (0,-7.65) {$\textcolor{green}{|S_8|=6 > 5}$};
			\node () at (-1,-8.5) {$8$};
			
			\draw[rec] (0,-11) rectangle (18,-10);
			\draw[grec] (0,-11) rectangle (8,-10);
			\node () at (20,-10.5) {$\piw(\cdot) \in (e^{-9},e^{-8}]$};
			\node () at (8,-9.75) {$V_9$};
			\node[anchor=west] () at (0,-9.65) {$\textcolor{green}{|S_9|=32 \geq \frac{1}{2}\alpha \cdot 6}$};
			\node () at (-1,-10.5) {$9$};
			
			\draw[rec] (0,-13) rectangle (18,-12);
			\draw[rrec] (0,-13) rectangle (17,-12);
			\node () at (20,-12.5) {$\piw(\cdot) \in (e^{-10},e^{-9}]$};
			\node () at (8,-11.75) {$V_{10}$};
			\node[anchor=west] () at (0,-11.65) {$\textcolor{red}{|S_{10}|=68 < \frac{1}{2}\alpha \cdot (6 + 32)}$};
			\node () at (-1,-12.5) {$10$};
			
			\draw[rec] (0,-15) rectangle (18,-14);
			\node () at (20,-14.5) {$\piw(\cdot) \in (e^{-11},e^{-10}]$};
			\node () at (8,-13.75) {$V_{11}$};
			\node[anchor=west] () at (0,-13.65) {$\textcolor{red}{|S_{11}|=0}$};
			\node () at (-1,-14.5) {$11$};

			\draw[brec] (-1.5,4.25) rectangle (-0.5,-1.25);
			\node[] at (-3,1.5) {\Large{\textcolor{blue}{$\mathcal{S}_1$}}};
			\node[] at (-2,3.5) {\Large{\textcolor{blue}{$a_1$}}};
			\node[] at (-2,-0.5) {\Large{\textcolor{blue}{$b_1$}}};

			\draw[nrec] (-1.5,-1.75) rectangle (-0.5,-7.25);
			\node[] at (-3,-4.5) {\Large{\textcolor{black}{$\mathcal{N}_1$}}};
			
			\draw[brec] (-1.5,-7.75) rectangle (-0.5,-11.25);
			\node[] at (-3,-9.5) {\Large{\textcolor{blue}{$\mathcal{S}_2$}}};
			\node[] at (-2,-8.5) {\Large{\textcolor{blue}{$a_2$}}};
			\node[] at (-2,-10.5) {\Large{\textcolor{blue}{$b_2$}}};
			
			\draw[nrec] (-1.5,-11.75) rectangle (-0.5,-15.25);
			\node[] at (-3,-13.5) {\Large{\textcolor{black}{$\mathcal{N}_2$}}};
			
	\end{tikzpicture}}
		\end{center}
		\caption{Illustration of the construction of $R$ for a given $S \subseteq V$, where $\alpha=e^2$.
			A new sequence starts if the size of the $S_i$ is at least as large as the size of the previous set, and a sequence ends as soon as $S_i$ is not larger than $\frac{1}{2} \alpha$ times of the size of all previous sets in that sequence.
		}
		\label{fig:illustration}
	\end{figure}
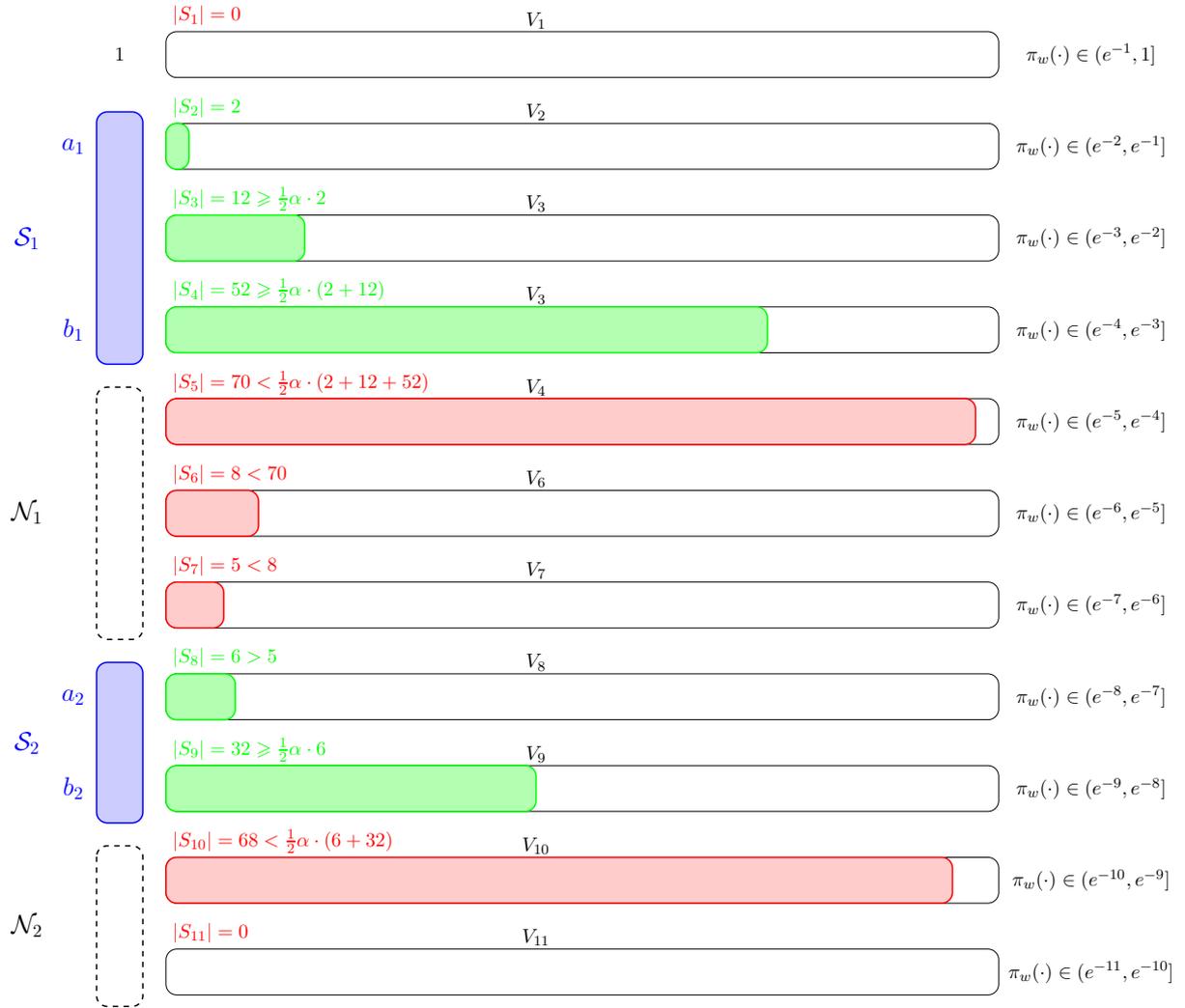

	The sets are growing inside the ``intervals'' $\good_\ell$:
	for all $\ell \in [L]$,
	we have
	\(
	\abs{S_{a_\ell}}
	>
	\abs{S_{a_\ell-1}}
	\)
	and
	\[
	\abs{S_{a_\ell+k}}
	\ge
	\tfrac12 \vex \abs{ S_{a_\ell} \cup S_{a_\ell+1} \cups S_{a_\ell+k-1} }
	\Qforall
	k \in [1, b_\ell - a_\ell].
	\]
	
	The first key fact about $\good$ is that it comprises an order-$1$ proportion of the stationary mass~of~$S$.
	
	\begin{lemma}
		\label{res:pi-calS_pi-S}
		For all $S \subseteq V$ with $\abs S \le \tfrac12 n$,
		we have
		\[
		\piw(\good)
		\ge
		\tfrac1{22}
		\piw(S).
		\]
	\end{lemma}
	
	\begin{Proof}
		We upper bound the contribution to $\piw(S)$ from \emph{outside} $\good$:
		set
		\[
		\mathcal N_\ell
		\cq
		S_{b_\ell+1} \cup S_{b_\ell+2} \cups S_{a_{\ell+1}-2} \cup S_{a_{\ell+1}-1}
		\Qfor
		\ell \in [L].
		\]
		
		Fix $\ell \in [L]$.
		The definition of $b_\ell$ used twice
		and the disjointness of the union
		implies that
		\[ 
		\abs{ S_{b_\ell+1} }
		\le
		\tfrac12 \vex \abs{ S_{a_\ell} \cups S_{b_\ell-1} \cup S_{b_\ell} } \\
		=
		\tfrac12 \vex \bigl( \abs{ S_{a_\ell} \cups S_{b_\ell-1} } + \abs{ S_{b_\ell} } \bigr)
		\le
		(1 + \tfrac12 \alpha) \abs{ S_{b_\ell} }.
		\]
		In turn, this immediately implies that
		\[
		\piw(S_{b_\ell+1})
		\le
		(1 + \tfrac12 \vex) \piw(S_{b_\ell}), 
		\]
		recalling that $\piw(u) < 1/e^i \leq \piw(v)$ whenever $u \in S_{i+1}$ and $v \in S_i$, for any $i$.
		
		Moreover, for any $i, k \in \mbn$, if $u \in S_{i+k}$ and $v \in S_i$, then
		\(
		\piw(u) / \piw(v) \le 1/e^{k-1}.
		\)
		Additionally, the map $k \mapsto \abs{S_{b_\ell+k+1}} $, where $k \in [0, a_{\ell+1} - b_\ell - 2]$, 
		is decreasing by definition of $a_{\ell+1}$.
		Together, these imply that
		\[
		\piw(S_{b_\ell+1 + k})
		\le
		1/e^{k-1} \piw(S_{b_\ell+1})
		\Qforall
		k \in [0, a_{\ell+1} - b_\ell - 2].
		\]
		This exponential decay in the stationary mass along the subsequence $\mathcal N_\ell$ implies that
		\[
		\piw(\mathcal N_\ell)
		\le
		\tfrac{e}{1-1/e}
		\piw(S_{{b_\ell}+1})
		\le
		\tfrac{e}{1-1/e} (1 + \tfrac12 \vex) \piw(S_{b_\ell})
		\le
		21 \piw(\good_\ell),
		\]
		recalling that $\vex = e^2$.
		The claim now follows readily:
		\[\ts
		\piw(S)
		=
		\sum_{\ell \in [L]}
		\bigl( \piw(\good_\ell) + \piw(\mathcal N_\ell) \bigr)
		\le
		\sum_{\ell \in [L]}
		22 \piw(\good_\ell)
		=
		22 \piw(\good).
		\qedhere
		\]
	\end{Proof}

	The second key fact is that each $S_{b_\ell}$ ($\ell \in [L]$) comprises an order-1 proportion of the stationary mass of $R_{\ell}$.
	
	\begin{lemma}
		\label{res:robust:pi(S_bl)>pi(mcS_l)}
		For
		all $S \subseteq V$ with $1 \leq \abs S \le \tfrac12 n$
		and
		all $\ell \in [L]$,
		we have
		\[
		\piw(S_{b_\ell})
		\ge
		\tfrac1{11}
		\piw(\good_\ell).
		\]
	\end{lemma}
	
	\begin{Proof}
		Fix $S \subseteq V$ and $\ell \in [L]$.
		For all $k \in [0, a_\ell - b_\ell]$, the following relations hold by definition of $b_\ell$ and $S_i$, namely the exponential decay of probabilities, respectively:
		\[
		\abs{S_{b_\ell}}
		\ge
		(\tfrac12 \vex)
		\abs{S_{b_\ell-1}}
		\ge
		\cdots
		\ge
		(\tfrac12 \vex)^k
		\abs{S_{b_\ell-k}}
		\Qand
		\min_{u \in S_{b_\ell}}
		\piw(u)
		\ge
		1/e^{k+1}
		\max_{v \in S_{b_\ell-k}}
		\piw(v).
		\]
		This implies that
		\[
		\piw(S_{b_\ell-k})
		\le
		\abs{S_{b_\ell-k}}
		\max_{v \in S_{b_\ell-k}}
		\piw(v)
		\le
		e
		(\tfrac12 \vex / e)^{-k}
		\abs{S_{b_\ell}}
		\min_{u \in S_{b_\ell}}
		\piw(u)
		\le
		e
		(\tfrac12 \vex / e)^{-k}
		\piw(S_{b_\ell}).
		\]
		Now since $\alpha=e^2>2e$, it follows $(\tfrac12 \vex / e) > 1$ which means that the upper bound above is exponentially decreasing in $k$. Hence,
		\begin{align*}\ts
			\piw(\good_\ell)
			=
			\sum_{k=0}^{a_\ell-b_\ell}
			\piw(S_{b_\ell-k})
			&\ts
			\le
			e \piw(S_{b_\ell})
			\sum_{k=0}^{a_\ell-b_\ell}
			(2e/\vex)^k
			\\&
			\le
			e (1 - 2e/\vex)^{-1}
			\piw(S_{b_\ell})
			\le
			11 \piw(S_{b_\ell}).
			\qedhere
		\end{align*}
	\end{Proof}
	
	Recall that we are interested in $Q_{P_w^{2K}}(S, S^c)$.
	We may decompose this as
	\[\ts
	Q_{P_w^{2K}}(S, S^c)
	\ge
	Q_{P_w^{2K}}(\good, S^c)
	=
	\sum_{\ell \in [L]}
	Q_{P_w^{2K}}(\good_\ell, S^c).
	\]
	Our next goal is to control $Q_{P_w^{2K}}(\good_\ell, S^c)$ for each $\ell$.
	The approach differs depending on $\vex \abs{\good_\ell}$.

	\begin{lemma}
		\label{res:robust:B2(S)-S}
		For all $S \subseteq V$ with $1 \leq \abs S \le \tfrac12 n$, and
		any $\ell \in [L]$,
		we have
		\[
		\abs{ \mcb_{2K}(\good_\ell) \setminus S }
		\ge
		\tfrac13
		\abs{\good_\ell}.
		\]
	\end{lemma}

	\begin{Proof}
		Fix $S$ and $\ell$.
		We split the proof into two cases according to the size of $\abs{\good_\ell}$.
		
		\textit{Case 1.}
		Suppose first that $\vex \abs{\good_\ell} \le \tfrac12 n$.
		We apply \cref{res:robust:ball-expansion} to get
		\[
		\abs{ \mcb_K(\good_\ell) }
		\ge
		\min\{ \vex \abs{\good_\ell}, \: \tfrac12 n \}
		=
		\alpha \abs{\good_\ell}.
		\]
		By the definitions of $a_\ell$ and $b_\ell$, we have
		\[
		\abs{S_{a_\ell-1}}
		<
		\abs{S_{a_\ell}}
		\le
		\abs{\good_\ell}
		\Qand
		\abs{S_{b_\ell+1}}
		\leq
		\tfrac12 \vex \abs{\good_\ell}.
		\]
		Next, since by \cref{lem:robust:lip-neighbours} for all $i \geq 1$,
		$
		\mcb_K(V_i)
		\subseteq
		V_{i-1} \cup V_i \cup V_{i+1},
		$
		it follows that
		\begin{align*}
			\abs{ \mcb_K(\good_\ell) \setminus S }
			&
			\ge
			\abs{ \mcb_K(\good_\ell) \setminus \good_\ell } - \abs{S_{a_\ell-1}} - \abs{S_{b_\ell+1}}
			\nonumber
			\\&
			\ge
			\abs{ \mcb_K(\good_\ell) } - \abs{\good_\ell} - \abs{S_{a_\ell-1}} - \abs{S_{b_\ell+1}}.
		\end{align*}
		Putting all these parts together, we deduce that
		\[
		\abs{ \mcb_K(\good_\ell) \setminus S }
		\ge
		(\vex - 1 - 1 - \tfrac12 \vex) \abs{\good_\ell}
		=
		(\tfrac12 \vex - 2) \abs{\good_\ell}.
		\]
		In particular, $\mcb_{2K}(\good_\ell) \supseteq \mcb_K(\good_\ell)$, so
		\(
		\abs{\mcb_{2K}(\good_\ell) \setminus S}
		\ge
		\abs{\mcb_K(\good_\ell) \setminus S}
		\ge
		\tfrac13\abs{\good_\ell},
		\)
		as $\vex = e^2$.

		\textit{Case 2.}
		Suppose next that $\vex \abs{\good_\ell} \ge \tfrac12 n$. Again, we apply \cref{res:robust:ball-expansion}:
		
		\[
		\abs{ \mcb_K(\good_\ell) }
		\ge
		\min\{ \vex \abs{\good_\ell}, \: \tfrac12 n \}
		=
		\tfrac12 n.
		\]
		We consider two subcases.
		Recall that $\abs S \le \tfrac12 n$ always.
		
		\textit{Case 2a.}
		Suppose first that $\abs{ \mcb_{2K}(\good_\ell) } \le \tfrac23 n$.
		Then, by \cref{res:robust:vertex-expansion},
		\[
		\abs{ \mcb_{K+k+1}(\good_\ell) \setminus \mcb_{K+k}(\good_\ell) }
		\ge
		\tfrac13 \psi
		\cdot
		\tfrac13 n
		\ge
		\tfrac16 \psi
		\abs{ \mcb_{K+k}(\good_\ell) }
		\Qforall
		k \in \{0, 1, ..., K\}.
		\]
		In particular,
		\[
		\abs{ \mcb_{K+k+1}(\good_\ell) }
		\ge
		(1 + \tfrac16 \psi)
		\abs{ \mcb_{K+k}(\good_\ell) }
		\Qforall
		k \in \{0, 1, ..., K\}.
		\]
		Iterating this,
		\[
		\abs{ \mcb_{2K}(\good_\ell) }
		\ge
		(1 + \tfrac16 \psi)^K
		\abs{ \mcb_K(\good_\ell) }
		\ge
		\tfrac12 e^{1/3} n,
		\]
		using the fact that $(1 + \tfrac16 x)^{2/\log(1+x)} \ge e^{1/3}$ uniformly in $x > 0$,
		and that $\abs{ \mcb_K(\good_\ell) } \ge \tfrac12 n$.~%
		Hence,
		\[
		\abs{ \mcb_{2K}(\good_\ell) \setminus S }
		\ge
		\abs{ \mcb_{2K}(\good\ell) } - \abs S
		\ge
		\tfrac12 (e^{1/3} - 1) n
		\ge
		(e^{1/3} - 1) \abs{\good_\ell}
		\ge
		\tfrac13
		\abs{\good_\ell}.
		\]
		\textit{Case 2b.}
		Finally, suppose that $\abs{ \mcb_{2K}(\good_\ell) } \ge \tfrac23 n$.
		Then, simply
		\[
		\abs{ \mcb_{2K}(\good_\ell) \setminus S }
		\ge
		\abs{ \mcb_{2K}(\good_\ell) } - \abs S
		\ge
		\bigl( \tfrac23 - \tfrac12 \bigr) n
		=
		\tfrac16 n
		\ge
		\tfrac13
		\abs{\good_\ell}.
		\qedhere
		\]
	\end{Proof}

	We are now able to lower bound the edge conductance $Q_{P_w^{2K}}(\good_\ell, S^c)$ in terms of $\piw(\good_\ell)$.
	
	\begin{lemma}
		\label{res:robust:Q(Sl)>pi(Sl)}
		For all $S \subseteq V$ with $\abs S \le \tfrac12 n$, and
		any $\ell \in [L]$,
		we have
		\[
		Q_{P_w^{2K}}(\good_\ell, S^c)
		\ge
		\tfrac1{90} d^{-2K}
		\piw(\good_\ell).
		\]
	\end{lemma}
	
	\begin{Proof} Using \cref{eq:weighted} and \cref{eq:lipschitz}, it follows that the minimal non-zero entry of $P_w$ is at least
		$
		1/(d \sigma).
		$
		Consequently, the minimal non-zero entry of $P_{w}^{2K}$ is at least $1/(d \sigma)^{2K} = e^{-1} d^{-2K}$. Hence,
		\[
		Q_{P_w^{2K}}(\good_\ell, S^c)
		\ge
		e^{-1} d^{-2K}
		\abs{ \mcb_{2K}(\good_\ell) \setminus S }
		\min_{u \in \good_\ell} \piw(u).
		\]
		The lower bound from \cref{res:robust:B2(S)-S} yields
		\(
		\abs{ \mcb_{2K}(\good_\ell) \setminus S }
		\ge
		\tfrac13
		\abs{\good_\ell}.
		\)
		Hence,
		\[
		Q_{P_w^{2K}}(\good_\ell, S^c)
		\ge
		\tfrac13 e^{-1} d^{-2K}
		\abs{\good_\ell}
		\min_{u \in \good_\ell} \piw(u).
		\]
		Recall that $\good_\ell \supseteq S_{b_{\ell}}$, and that $\piw(u')$ varies by a factor at most $e$ over $u' \in S_{b_{\ell}}$.
		Hence,
		\[
		\abs{\good_\ell}
		\min_{u \in \good_\ell} \piw(u)
		\ge
		\abs{S_{b_{\ell}}}
		\min_{u' \in S_{b_{\ell}}} \piw(u')
		\ge
		e^{-1}
		\piw(S_{b_\ell}),
		\]
		since $\arg\min_{u \in \good_\ell} \piw(u) \in S_{b_\ell}$, by definition.
		By \cref{res:robust:pi(S_bl)>pi(mcS_l)}, $\piw(S_{b_\ell}) \ge \tfrac1{11} \piw(\good_\ell)$.
		Hence,
		\[
		Q_{P_w^{2K}}(\good_\ell, S^c)
		\ge
		\tfrac1{33} e^{-1} d^{-2K}
		\piw(\good_\ell)
		\ge
		\tfrac1{90} d^{-2K}
		\piw(\good_\ell).
		\qedhere
		\]
	\end{Proof}
	
	We now put all these ingredients together to prove the desired edge conductance bound.
	
	\begin{Proof}[Proof of \cref{res:robust}]
		Recall that $\alpha \cq e^2$. Further, consider any $S \subseteq V$ with $1 \leq \abs S \le \tfrac12 n$.
		Summing over $\ell \in [L]$, and applying \cref{res:robust:Q(Sl)>pi(Sl)},
		\[\ts
		Q_{P_w^{2K}}(S, S^c)
		\ge
		\sum_{\ell \in [L]}
		Q_{P_w^{2K}}(\good_\ell, S^c)
		\ge
		\tfrac1{90} d^{-2K}
		\sum_{\ell \in [L]}
		\piw(\good_\ell)
		=
		\tfrac1{90} d^{-2K}
		\piw(\good).
		\]
		Finally, for such $S$, we use \cref{res:pi-calS_pi-S} to obtain
		\[
		Q_{P_w^{2K}}(S, S^c) / \piw(S)
		\ge
		\tfrac1{90 \cdot 22}
		d^{-2K}
		\ge
		\tfrac1{2000}
		d^{-2K}.
		\]
		
		Now choose an arbitrary $S \subseteq V$ with $S \notin \{\varnothing, V\}$.
		Then,
		\[
		\tfrac1{2000}
		d^{-2K}
		\le
		\begin{cases}
			Q_{P_w^{2K}}(S, S^c) / \piw(S)
			&\text{if}\quad
			\abs S \le \tfrac12 n,
			\\
			Q_{P_w^{2K}}(S^c, S) / \piw(S^c)
			&\text{if}\quad
			\abs S > \tfrac12 n.
		\end{cases}
		\]
		However, $Q_{P_w^{2K}}(S, S^c) = Q_{P_w^{2K}}(S^c, S)$ for all $S \subseteq V$.
		Hence, in either case,
		\[
		\widetilde \Phi_{P_w^{2K}}(S)
		\cq
		\frac{Q_{P_w^{2K}}(S, S^c)}{\piw(S) \piw(S^c)}
		\ge
		\tfrac1{2000}
		d^{-2K};
		\]
		note that $\widetilde \Phi_{P_w^{2K}}(S) = \widetilde \Phi_{P_w^{2K}}(S^c)$.
		Minimizing over $S \subseteq V$ with $S \notin \{\varnothing, V\}$ gives
		\[
		\widetilde \Phi_{P_w^{2K}}
		\cq
		\min_{S \subseteq V : S \notin \{\varnothing, V\}}
		\widetilde \Phi_{P_w^{2K}}(S)
		\ge
		\frac1{2000}
		d^{-2K}.
		\]
		Finally, $\Phi_{P_w^{2K}} \ge \tfrac12 \widetilde \Phi_{P_w^{2K}}$ since $\piw(S) < \tfrac12$ implies $1/\piw(S^c) \le 2$.
		
		Cheeger's inequality \cite[Lemma~3.3]{JS:approx-counting} (see also \eqref{eq:cheeger}) can be used to convert the isoperimetric bound $\Phi$ into the claimed bound on the spectral gap, using also the fact that the eigenvalues of $P_w^{2K}$ are those of $P$ raised to the $2K$-th power.
		This completes the proof.
	\end{Proof}
	
	Finally, we will show how to derive \cref{res:robustlite} from \cref{res:robust}.  
	\begin{proof}[Proof of \cref{res:robustlite}]  
		Following the assumptions of \cref{res:robustlite}, let $P$ be the transition matrix of a simple random walk on a $d$-regular graph and $\gamma$ be its spectral gap. Further, let $w:E \rightarrow \mathbb{R}+$ be a Lipschitz-weighting with parameter $\beta \leq 1 + \gamma/32$. Since for any graph, $0<\psi \leq 2$, and also $2/\log 3 \geq 3/2$, we conclude 
		\begin{equation}\label{eq:K} \left\lceil \frac{2}{\log(1+\psi)} \right\rceil \leq 2 \cdot \frac{2}{\log(1+\psi)}  = \frac{4}{\log(1+\psi)} .\end{equation}
		
		We will apply \cref{res:robust} with $\psi:=\Psi_{G}$, to do so we need to verify that $\sigma \geq \beta$, and then derive the form of the bound in the statement.  Indeed, by the definitions of $\sigma$ and $K$, 
		\[
		\sigma = e^{1/(2K)} = e^{1/(2 \lceil 2/\log(1+\psi) \rceil )} \stackrel{\eqref{eq:K}}{\geq} (1+\psi)^{1/8} = (1+\psi)^{1/\psi \cdot \psi/8} \stackrel{(*)}{\geq} 3^{\psi/16},
		\]
		where $(*)$ holds since $f(z)=(1+z)^{1/z}$ is decreasing for $0 < z \leq 2$. On the other hand,
		\[
		\beta  \leq 1+\gamma/32 
		\leq \exp(\gamma/32)  \stackrel{\cref{eq:vertexspectral}}{\leq} \exp(\Psi_G/16)=\exp(\psi/16) \leq 3^{\psi/16}.
		\]
		Thus, $\beta \leq \sigma$, and so we can apply~\cref{res:robust}  under the conditions of \cref{res:robustlite}. 
		
		Finally, $\log(1+\psi) \geq \frac{\psi}{2+\psi} \geq \psi/4 \geq \gamma/4 $ as $0<\psi\leq 2$ and by \eqref{eq:vertexspectral}, hence $4K \leq 16/\gamma$ by \eqref{eq:K}. Thus, \cref{res:robust} gives $ \gamma_{P_w} \ge 10^{-8} \cdot 
		d^{-4K} 
		\geq 10^{-8} \cdot d^{-16/\psi} $, as claimed.
	\end{proof}
	
	\subsection{Upper Bound on the Spectral Gap}
	
	The next result shows that it is essential to have an upper bound on $\beta$, that depends on some expansion measure, as otherwise the edge conductance may be polynomially small in $n$ for certain edge-weightings. Note that this is stated and proved using edge conductance, and we obtain \cref{pro:lowerspec} in terms of the spectral gap by applying Cheeger's inequality \cref{eq:cheeger}.  
	\begin{proposition}\label{pro:lower}
		Let $G$ be a $d$-regular graph with diameter $D \geq 4$. Then for any $\beta > 1$, there exists a $\beta$-Lipschitz edge-weighting $w:E \rightarrow \mathbb{R}^{+}$ such that the chain $P_w$ induced by $w$ satisfies 
		\[
		\Phi_{P_w} \leq \min \left\{ d^{\lfloor D/2 \rfloor-1},n \right\} \cdot \beta^{-\lfloor D/2 \rfloor + 3}.
		\]
	\end{proposition}
	
	In particular this implies that for any bounded-degree $d$-regular graph $G$, as $ D=\Omega(\log n)$, we can find edge-weightings satisfying the Lipschitz condition with constant $\beta > d$ resulting in an edge conductance which is at most $n^{-\Omega(\log \beta)}$. Using the well-known connection between mixing times of (lazy random walks) and hitting times~\cite[Theorem~10.22]{levin2009markov}, this also implies that the maximum hitting time on these weighted graphs is at least $n^{\Omega(\log \beta)}$.

	\begin{Proof}[Proof of \cref{pro:lower}]
		Let $u,v\in V$ with $\dist(u,v)=D \geq 4$. For any $\beta\geq 1 $, we define the weight of any edge $\{x,y\}$ as follows:
		\begin{align}
			w(x,y) = \beta^{-\dist(\{x,y\},\{u,v\})}. \label{eq:weight_def}
		\end{align}
		Clearly, this edge-weighting satisfies the $\beta$-Lipschitz condition \eqref{eq:lipschitz}. Define \[B_1:=\left\{ w \in V \colon \dist(u,w) \leq \lfloor D/2 \rfloor - 1 \right\},\quad \text{and} \quad B_2:=\left\{ w \in V \colon \dist(v,w) \leq \lfloor D/2 \rfloor - 1 \right\}.\] 
		Since $D \geq 4$, it follows not only that $u \in B_1$ and $v \in B_2$, but also that $B_1$ and $B_2$ are disjoint. Hence, $\piw(B_1) \leq 1/2$ or $\piw(B_2) \leq 1/2$ (or possibly both). Without loss of generality, assume $\piw(B_1) \leq 1/2$. Since $u \in B_1$, we have $\piw(B_1) \geq \piw(u)$.
		Furthermore, let \[B:=\left\{ w \in V \colon \dist(u,w) = \lfloor D/2 \rfloor - 1 \right\}.\] Note that by the triangle inequality, any vertex $x$ with $\dist(x,u) = k \leq \lfloor D/2 \rfloor - 1$ must satisfy $\dist(x,v) > \lfloor D/2 \rfloor$. Hence the weight of any edge $\{x,z\}$ is equal to $\beta^{-\dist(\{x,z\},u)}$. Therefore, by \cref{eq:weight_def}, any vertex $x$ with $\dist(u,x)=k \leq \lfloor D/2 \rfloor - 1$ must also satisfy
		\[
		w(x) \leq d \cdot \beta^{-k+1},
		\]
		and vertex $u$ trivially satisfies $w(u) = d \cdot \beta^{-1}$,
		and so
		\[
		\frac{\piw(x)}{\piw(u)} = \frac{w(x)/W}{w(u)/W} \leq  \beta^{-k+2}.
		\]
		Hence it follows that
		\begin{equation}\label{eq:piupper}
			\piw(B) \leq \max_{x \in B} \pi(x) \cdot |B|
			\leq 
			\piw(u) \cdot \beta^{-\lfloor D/2 \rfloor +3} \cdot \min \left\{ d^{\lfloor D/2 \rfloor - 1},n \right\}.
		\end{equation}
		This implies for the ergodic flow,
		\[
		Q_{P_w}(B_1,B_1^c) = \sum_{x \in B_1, y \not\in B_1} 
		\piw(x) \cdot P_w(x,y) = \sum_{x \in B} \piw(x) \sum_{y \not\in B_1} P_w(x,y) \leq \piw(B),
		\]
		and therefore, as $\piw(B_1) \leq 1/2$, we have
		\[
		\Phi_{P_w} \leq \frac{Q_{P_w}(B_1,B_1^c)}{\piw(B_1)} 
		\leq \frac{\piw(B)}{\piw(B_1)}.\] Hence,  inserting the bound from \eqref{eq:piupper} gives the claimed bound:
		\[\Phi_{P_w}\leq \frac{\piw(u) \cdot \min \left\{ d^{\lfloor D/2 \rfloor-1},n \right\} \cdot \beta^{-\lfloor D/2 \rfloor+3}}{\piw(u)} = \min \left\{ d^{\lfloor D/2 \rfloor-1},n \right\} \cdot \beta^{-\lfloor D/2 \rfloor+3}.
		\qedhere\] 
	\end{Proof}

	\section{Upper Bounds on the Cover Times of the Time-Biased Walk}\label{sec:cover}

	Given a graph $G$ and a non-empty vertex subset $U\subseteq V(G)$, we use a specific set of edge weights in order to define a so-called \textit{biased random walk towards the set $U$}.  Specifically, for some $0\leq \biasname< 1$ and for every edge $e=\{u,v\} \in E(G)$, the weight of the edge is defined as \begin{equation}
		\label{edgeweights}
		w(u,v) = (1-\biasname)^{ \max\{ \dist(u,U), \,\dist(v,U) \}}.\end{equation}  We denote the transition matrix of this walk by $Q:=Q(U,\biasname)$, that is \[Q(u,v)=\begin{cases}
		w(u,v)/w(u) & \text{ if } \{u,v\} \in E, \\
		0 & \text{ otherwise }
	\end{cases},  \] where $w(u)=\sum_{\{u,v\}\in E}w(u,v)$. Generally $\biasname$ will be upper bounded by some small $\eps \leq 1$. Observe that for any $u,v,z \in V$ such that  $u \sim v$ and   $v \sim z$, we have
	\begin{equation}\label{eq:x-walk-lipschitz}
		1-\biasname \leq \frac{w(u,v)}{w(v,z)} \leq \frac{1}{1-\biasname};
	\end{equation}
	thus, this weighting scheme is $\beta$-Lipschitz \eqref{eq:lipschitz} for $\beta= \frac{1}{1-\biasname}\geq 1$.

	The following result is similar to \cite[Prop.\ 4.5]{ETB}, but we include a proof for completeness. 
	\begin{proposition}\label{weightedetb}
		Let $ \biasname \in [0,1)$ be arbitrary, and $G$ be any graph with minimum degree at least $3$, and $U \subseteq V$. Let $P$ be the simple random walk on $G$ and $Q:=Q(U,\biasname)$  be as given by \eqref{edgeweights}. Then, for any $\biasname\leq \eps \leq 1 $, there exists a stochastic matrix $B$ such that $(1-\eps)P+\eps B = Q$.  
	\end{proposition}
	
	\begin{Proof}
		
		If we  set $B(u,v) = \frac{1}{\eps}\left(Q(u,v) -  (1-\eps)/d(u)    \right)$ for any vertices $u,v \in V$, then \[(1-\eps)\cdot 1/ d(u) + \eps \cdot B(u,v) = Q(u,v)\quad \text{ for all }u,v\in V,\] as desired. We must now show that this choice for $B$ gives a stochastic matrix. Firstly, \[\sum_{v\sim u } B(u,v) = \frac{1}{\eps}\sum_{v\sim u } \left(Q(u,v) -  (1-\eps)/d(u)    \right) = \frac{1}{\eps}\left(1 - (1-\eps) \right)= 1,\] as $Q$ is a stochastic matrix.
		So, it remains to show that $B(u,v) \ge 0$ for all $u,v\in V$.

		Observe that for any $u\in V$, all edges adjacent to $u$ have one of at most two distinct weights  $w_1$ and $w_2$. Since $Q(u,v)$ is proportional to $w(u,v)$, we can assume without loss of generality that $w_1=1$ and $w_2=1-\biasname$. Say there are $b \le d(u)$ edges with weight $w_2$. Under these~assumptions 
		\[
		w(x)=(d(u)-b)\cdot 1 + b\cdot (1-\biasname) = d(u) -b\biasname.
		\]
		If all these edge have the same weight then we are done as $Q(u,v)=1/d(u) \in [0,1]$ and thus $B(u,v)=1/d(u) \geq 0$ for all $u,v\in V$. Hence, we assume that $1\leq b \leq d(u)-1$.
		Note that it suffices to prove the result for $\eps =  \biasname$ as if $\eps$ is larger than this, then the walk can always just not use the full power of the bias. 
		Finally,
		\[
		B(u,v) = \frac{1}{\eps}\left(\frac{w(u,v)}{w(u)} -  \frac{1-\eps }{d(u)}    \right) \geq  \frac{1}{\biasname}\left(\frac{1-\biasname}{d(u)-\biasname} -  \frac{1-\biasname }{d(u)}    \right)\geq 0.
		\qedhere
		\]  
	\end{Proof}
	
	We now prove a lower bound showing that the stationary probability of any set $U$ can be ``boosted'' by the edge-weighting of $Q(U,\biasname)$. The proof follows some ideas found in the proof of \cite[Thm.~4.6]{ETB}; note that the next lemma assumes $d \geq 3$, since the only connected, $2$-regular graphs are cycles, which are not expanders but are easily shown to have linear cover time.
	
	\begin{lemma}\label{lem:stationary_boost}
		Let $G$ be any $d$-regular graph with $d \geq 3$. Then, for any non-empty $U\subseteq V(G)$, $u \in U$, and $0 \leq \biasname \leq 1/3$, the walk given by $Q:=Q(U,\biasname)$ from \eqref{edgeweights} satisfies 
		\[
		\pi_{Q}(u) \geq \frac{1}{2d|U|} \cdot  \left( \frac{|U|}{n} \right)^{1+\log(1-\biasname)/\log d}.
		\]
	\end{lemma}
	\begin{Proof}
		To prove the claim, recall that $\pi_{Q}(u) = \frac{w(u)}{W}$. By \cref{edgeweights}, for every $u \in U$ and $v \sim u$, we have $w(u,v) \in [1-\biasname,1]$, and consequently,
		$w(u) \in [(1-\biasname) \cdot d, d]$, as $G$ is $d$-regular. Likewise for any vertex $v$ at distance $i$ from $u$ we have $w(v)\leq d\cdot (1-\biasname)^i $. Thus, if we let $U_i $ be the set of vertices at distance $i\geq 0$ from $U$, then by \eqref{eq:weights} 
		\[W:=\sum_{x \in V}
		w(x) = \sum_{i=0}^\infty \sum_{x \in U_i} w(x) \leq \sum_{i=0}^\infty \sum_{x \in U_i}d\cdot (1-\biasname)^i .\]
		
		Since there are $n$ vertices in total and the weights $d\cdot(1-\biasname)^i$ are decreasing in $i\geq 0$, the sum on the RHS above is maximized when all vertices not in $U$ are as close to $U$ as possible. Observe also that $|U_i| \leq |U|\cdot d^i$, and thus $r := \lceil \log_{d} (n/|U|) \rceil$ is the maximum distance a vertex $x\notin U$ can be from $U$ under this assumption, which maximizes $W$. Consequently,  
		\[ 
		W  \leq \sum_{i=0}^{r} |U| \cdot d^{i}\cdot d\cdot  (1-\biasname)^{i} = d|U| \cdot \sum_{i=0}^r \left( d (1-\biasname) \right)^{i}.
		\]
		Now, since $d \geq 3$ and $\biasname \leq \frac{1}{3}$, we have $d(1-\biasname) \geq 2$. Thus, from the above, 
		\begin{align*}
			W
			&
			\le
			2 \cdot d|U|  \cdot (d(1-\biasname))^{r}
			\\&
			\le
			2 \cdot d|U|  \cdot (d(1-\biasname))^{\log_{d} (n/|U|) +1}
			\\&
			=
			2 \cdot d|U|  \cdot d(1-\biasname)\cdot d^{\log_{d} (n/|U|)}\cdot (e^{\log(1-\biasname)})^{\frac{\log (n/|U|)}{\log d}}
			\\&
			=
			2\cdot d|U| \cdot  d(1-\biasname)\cdot \frac{n}{|U|} \cdot  \Bigl( \frac{n}{|U|} \Bigr)^{\log(1-\biasname)/\log d}
			\\&
			=
			2d^2 |U| \cdot (1-\biasname) \cdot \Bigl( \frac{n}{|U|} \Bigr)^{1+\log(1-\biasname)/\log d}.
		\end{align*}
		Since any edge with an endpoint in $U$ has weight at least $1-\biasname$,  it follows that for every $u \in U$,
		\[
		\pi_Q(u)
		\ge
		\frac{d(1-\biasname)}{W}
		\ge
		\frac{1}{2d|U|} \cdot \left( \frac{|U|}{n} \right)^{1+\log(1-\biasname)/\log d}.
		\qedhere
		\]
	\end{Proof}

	We will make use of the following bound, whose proof we include for completeness. 
	\begin{proposition}[folklore]\label{prop:diam}
		For any $d$-regular, connected graph $G$, the diameter satisfies
		\[
		D \leq 8 \cdot \frac{ \log n}{\Psi_G}.
		\]
	\end{proposition}
	\begin{Proof}
		As $\Psi_G \leq 2$ and $D \leq n-1$, the claim holds trivially for any $n \leq 7$, so we may assume $n \geq 8$ in what follows.
		Take any two vertices $u,v \in V$. Then by \cref{res:robust:ball-expansion}, it follows that for any $k \geq 0$,
		\[
		|\mathcal{B}_{k}(u) |
		\geq \min\{(1+\Psi_G)^k,n/2\},
		\]
		and the same inequality holds if $u$ is replaced by $v$. Choosing $k := \lceil \log_{1+\Psi_G}(n/2) \rceil$ implies $\mathcal{B}_{k}(u) $ and $\mathcal{B}_{k}(v) $ both contain at least $n/2$ vertices, and thus $\dist(u,v) \leq 2 k + 1$. Since $u,v$ are arbitrary,
		\begin{align*}
			D \leq 2 \cdot \lceil \log_{1+\Psi_G}(n/2) \rceil + 1 &\stackrel{(a)}{\leq} 4 \cdot  \log_{1+\Psi_G}(n/2) + 1 \\
			&= 4 \cdot \frac{ \log (n/2) }{ \log(1+\Psi_G)} + 1 \\
			&\stackrel{(b)}{\leq} 8 \cdot \frac{ \log n - 0.5 }{ \Psi_{G} } + 1 \\
			&\stackrel{(c)}{\leq} 8 \cdot \frac{\log n}{\Psi_G}
		\end{align*}
		where $(a)$ holds as $\log_{1+\Psi_G}(n/2) \geq \log_{3}(4) > 1$ (since $n \geq 8$ and $\Psi_G \leq 2$); $(b)$ holds as $\log(1 + \Psi_G) \ge \Psi_G/2$ for any $\Psi_G \in [0,2]$, and $\log(2) > 0.5$; and $(c)$ holds as $\Psi_G \leq 2$. 
	\end{Proof}
	
	The next result is the most involved component of the cover time proof, and it is also in this proof where we apply our lower bound on the spectral gap (\cref{res:robust}). While this result holds in greater generality, it might be useful for the reader to think of $d$ as being constant.
	\begin{proposition}\label{lem:keylemma_cover}
		Let $G$ be any $d$-regular graph with $d \geq 3$, $\Psi_{G} \geq 49\ln d/\ln n$, and $n$ sufficiently large. Fix any non-empty $U\subseteq V(G)$ and let $\biasname \in [0, 1- e^{- \Psi_G/32}]$. Then, for the walk $Q:=Q(U,\biasname)$ from \eqref{edgeweights},  the expected time taken until at least half of the vertices in $U$ are visited is at most
		\[
		\kappa\cdot \left( \frac{n}{|U|} \right)^{\log(1-\biasname)/\log d} \cdot n,
		\]
		where $\kappa:=\kappa(d,\Psi_G) = 10^{10} \cdot d^{26/\Psi_G}> 0$.
	\end{proposition}

	\begin{Proof} 
		For all time-steps $t \geq 0$, we will employ the walk $Q:=Q(U,\biasname)$ from
		\cref{edgeweights} for our given $\biasname \in (0, 1- e^{- \Psi_G/32}]$ and non-empty subset $U$. Note that $Q$ is fixed at time $0$ and does not change in future rounds, even if vertices from $U$ are visited. By \eqref{eq:x-walk-lipschitz} this edge-weighting is $\beta$-Lipschitz with
		\[\beta:= (1-\biasname)^{-1}\geq 1\]

		First recall from \cref{lem:pibound} that for any two vertices $u,v \in V$ at distance at most $i$ we have $\pi_{Q}(v) \leq \beta^{2i} \cdot \pi_{Q}(u)$.
		By an averaging argument, at least one vertex $w \in V$ must satisfy $\pi(w) \leq 1/n$, and thus as $D \leq 8(\log n)/\Psi_G$ by \cref{prop:diam} we have  
		\begin{equation}\label{eq:maxstat}
			\max_{u \in V} \pi_{Q}(u) \leq  \left(\beta^{2} \right)^{8(\log n)/\Psi_G} 
			\cdot \pi_Q(w) \leq  \beta^{16(\log n)/\Psi_G} \cdot 1/n. \end{equation}
		Due to the precondition $\biasname\leq 1- e^{- \Psi_G/32}$, $\beta$ satisfies  \begin{equation}\label{eq:walkbeta}\beta := (1-\biasname)^{-1} \leq  e^{\Psi_G/32}.\end{equation}
		Thus, from \eqref{eq:maxstat} and \eqref{eq:walkbeta},	\begin{equation}\label{eq:pi_max}\max_{u \in V} \pi_{Q}(u)\leq  ( e^{\Psi_G/32})^{16(\log n)/\Psi_G} \cdot 1/n = n^{-1/2}.\end{equation}
		By the same averaging argument there is at least one vertex $z \in V$ with $\pi_Q(z) \geq 1/n$, and so 
		\begin{equation}
			\min_{u \in V} \pi_{Q}(u) \geq  \left(\beta^{2} \right)^{-8(\log n)/\Psi_G}  \cdot \pi_{Q}(z) \geq n^{-3/2}. \label{eq:pi_min}
		\end{equation}

		For the remainder of the proof we work with the ``lazy'' version of $Q$, given by $\tilde{Q}:=(Q+I)/2$, where $I$ denotes the identity matrix. Note that $\pi_{\tilde{Q}}=\pi_{Q}$ and also $\gamma_{\tilde{Q}} =1-\lambda_2(\tilde{Q})$ since all eigenvalues of $\tilde{Q}$ are in $[0,1]$. The reason for this is that we wish to apply some convergence results. Note that it suffices to prove this lemma for $\tilde{Q}$ as the time it takes $\tilde{Q}$ to first visit $U$ is at least that of $Q$, since any two trajectories can be coupled by removing the lazy steps.   
		
		In the following, we divide the walk given by $\tilde{Q}:=\tilde{Q}(U,\biasname)$ into consecutive epochs of $2 \sqrt{n}$ time-steps; the first $\sqrt{n}$ steps are used for mixing and the second $\sqrt{n}$ steps for covering vertices in $U$. 
		Since $\tilde{Q}$ is lazy we can employ the well-known inequality \cite[Equation~12.13]{levin2009markov}:
		\begin{equation}\label{eq:JS}
			\left| \tilde{Q}^t(w,u) - \pi_{Q}(u) \right| \leq \sqrt{ \frac{\pi_Q(u)}{\pi_Q(w)} } \cdot \lambda_2(\tilde{Q})^{t},
		\end{equation}
		where $w$  is the (arbitrary) start vertex of the random walk. We will apply~\cref{res:robust}, with $K := \lceil 2/\log(1+\Psi_G) \rceil$, where in order to apply this we need to verify that the $\beta$-Lipschitz edge weighting we use is $\sigma$-Lipschitz, that is $\beta \leq \sigma$. This is indeed the case:
		\begin{align*}      
			\sigma = e^{1/(2K)} = e^{1/(2 \lceil 2/\log(1+\Psi_G) \rceil )} \overset{(a)}{\geq} e^{\log(1+\Psi_G)/8}  \overset{(b)}{\geq} e^{\Psi_G/16}  \overset{\eqref{eq:walkbeta}}{\geq} \beta, 
		\end{align*}
		where $(a)$ holds since for any graph, $0<\Psi_G  \leq 2$, and also $2/\log 3 \geq 3/2$, which implies $\left\lceil \frac{2}{\log(1+\Psi_G)} \right\rceil \le 4/\log(1+\Psi_G)$, and $(b)$ holds since $\log(1+\Psi_G) \geq \Psi_G/2 $ as $ 0<\Psi_G\leq   2$.    Now using $\tilde{Q}=\frac{1}{2} Q + \frac{1}{2} I$ and then applying~\cref{res:robust}, with $K := \lceil 2/\log(1+\Psi_G) \rceil$, we have 
		\begin{align}
			\lambda_{2}(\tilde{Q}) = 1 - \gamma_{\tilde{Q}} = 1 - \gamma_{Q}/2 \leq 1 - \frac{1}{2}\cdot \frac{1}{10^8 
			} \cdot d^{-4 K} =: \rho. \label{eq:rho_def}
		\end{align}
		By the conditions  $49\ln d/\ln n \leq \Psi_{G} \leq 2$, and since $ \ln(1+x) \geq x/2$ for all $0\leq x\leq 2$, we have \[K = \left\lceil \frac{2}{\ln(1+\Psi_G)} \right\rceil \leq \left\lceil \frac{4}{\Psi_G} \right\rceil \leq  \frac{3}{2}\cdot \frac{4}{\Psi_G} \leq \frac{6\ln n}{49\ln d} .  \] Hence, using $1-z \leq \exp(-z)$, it follows that for large $n$, 
		\[\rho^{\sqrt{n}} \leq \exp\left(-\frac{d^{-4 K}}{2\cdot 10^8} \cdot \sqrt{n} \right) \leq \exp\left(-\frac{ d^{-\frac{24\ln n}{49\ln d}} }{2\cdot 10^8} \cdot \sqrt{n} \right)  \leq \exp\left(-n^{1/100}\right) .  \] 
		For any $t \in [\sqrt{n},2 \sqrt{n})$, we use \cref{eq:JS} and apply \eqref{eq:rho_def} as well as the lower and upper bounds on $\pi$ in \eqref{eq:pi_max} and  \eqref{eq:pi_min} to conclude that for any pair of vertices $u,w \in V$, and large $n$, 
		\[
		\tilde{Q}^t(w,u) \geq \pi_{Q}(u) - \sqrt{ \frac{n^{-1/2}}{n^{-3/2}} } \cdot \rho^{\sqrt{n}} \geq  \pi_{Q}(u) -\sqrt{n}\cdot\exp\left(-n^{1/100}\right) \geq   \frac{1}{2}\cdot \pi_{Q}(u) .
		\]

		For any vertex $u \in U$, let $N_u$ denote the number of visits to $u$ by the walk $\tilde{Q}:=\tilde{Q}(U,\biasname)$ during the time-interval $[\sqrt{n},2\sqrt{n})$. Then, for large $n$ and an arbitrary start vertex $w\in V$ at time $0$, 
		\[
		\Ex{ N_u } = \sum_{s=\sqrt{n}}^{2\sqrt{n}-1} \frac{1}{2} \cdot \pi_{Q}(u) 
		\geq \frac{\sqrt{n}}{2} \cdot \frac{1}{|U|}   \cdot \frac{1}{2d} \cdot \left( \frac{|U|}{n} \right)^{1+\log(1-\biasname)/\log d},
		\]
		where the last inequality follows by \cref{lem:stationary_boost},
		which in turn uses the precondition $\biasname \leq 1/3$, which is satisfied since $ \biasname\leq  1- e^{- \Psi_G/32}\leq 1- e^{-2/32}<1/3$.
		
		Let us now consider the number of returns to a fixed but arbitrary vertex $u \in U$. 
		We will apply inequality \eqref{eq:JS} again, but this time there is some cancellation as the start/finish vertices are both $u$, that is
		\[
		\left| \tilde{Q}^t(u,u) - \pi_{Q}(u) \right| \leq \lambda_2(\tilde{Q})^{t}
		\]
		Using \cref{eq:rho_def} and the upper bound on $\pi_{Q}(u)$ from \cref{eq:pi_max}, we have
		\[
		\tilde{Q}^t(u,u) \leq \max_{u \in V} \pi_{Q}(u) + \lambda_2(\tilde{Q})^{t} \leq n^{-1/2} + \rho^{t}.
		\]
		Hence by summing,
		\[
		\sum_{s=0}^{\sqrt{n}-1} \tilde{Q}^s(u,u) \leq 1 + \frac{1}{1-\rho}.
		\]
		Since
		\[
		\Ex{ N_u } = \Pr{ N_u \geq 1} \cdot \Ex { N_u \, \mid \, N_u \geq 1},
		\]
		and
		\[
		\Ex { N_u \, \mid \, N_u \geq 1} \leq \sum_{s=0}^{\sqrt{n}-1} \tilde{Q}^s(u,u) \leq 1 + \frac{1}{1-\rho},
		\]
		it follows that for every vertex $u \in U$, and start vertex $w\in V$ we have
		\[
		\Pr{ N_u \geq 1 } = \frac{ \Ex{ N_u }}{\Ex { N_u \, \mid \, N_u \geq 1} } \geq \frac{\frac{\sqrt{n}}{4d|U|} \cdot \left( \frac{|U|}{n} \right)^{1+\log(1-\biasname)/\log d}}{ 1 + \frac{1}{1-\rho}} =: p.
		\]
		Now, let $U_t$ be the vertices left unvisited in $U$ after $t$ rounds; so $U_0=U$. Then,
		\begin{align*}
			\Ex{\left.|U_{2 \sqrt{n}}| ~\right|~ U_0 } &= \sum_{u \in U_0} \Pr{ N_u = 0} \leq \sum_{u \in U_0} \left(1 - p \right) = \left(1-p \right) \cdot |U_0|.
		\end{align*}
		More generally, for any $k \geq 1$,
		\begin{align*}
			\Ex{\left.|U_{k \cdot (2 \sqrt{n})}|~\right|~U_{(k-1) \cdot (2 \sqrt{n})}}
			&\leq \sum_{u \in U_{(k-1) \cdot (2 \sqrt{n}) }} (1-p) = (1-p) \cdot |U_{(k-1) \cdot (2 \sqrt{n})}|.
		\end{align*}
		By iterating such phases of length $2 \cdot \sqrt{n}$ and using the tower law of expectations, we conclude 
		\begin{align*}
			\Ex{\left.|U_{k \cdot (2 \sqrt{n})}|~\right|~U_0}  &\leq (1-p)^{k} \cdot |U_0|.
		\end{align*}
		Hence taking $k:=\lceil 1/p \rceil$ phases, it follows that \[ \Ex{\left.|U_{k \cdot (2 \sqrt{n})}|~\right|~U_0}  \leq e^{-1} \cdot |U_0|.\] Since $U=U_0$, by Markov's inequality,
		\[
		\Pr{ |U_{ \lceil 1/p \rceil \cdot (2 \sqrt{n})}| \geq |U|/2 } \leq \frac{2}{e}.
		\]
		Therefore the time $\tau$ until half of the set $U$ is visited is stochastically smaller than $(2 \sqrt{n}) \cdot X\cdot \lceil 1/p \rceil$, where   $ X\sim \mathsf{Geo}(1-\frac{2}{e})$ is a geometric random variable with success probability $1-\frac{2}{e}> 1/4$. Thus,  
		\begin{align*}
			\Ex{\tau} &\leq (2 \sqrt{n}) \cdot  \frac{1}{1-\frac{2}{e}} \cdot \left\lceil \frac{1 + \frac{1}{1-\rho}}{\frac{\sqrt{n}}{4d|U|} \cdot \left( \frac{|U|}{n} \right)^{1+\log(1-\biasname)/\log d}} \right\rceil   \\
			&< (8\sqrt{n}) \cdot\left( \frac{ 1 +2\cdot 10^8  \cdot d^{4K} }{\frac{\sqrt{n}}{4d|U|} \cdot \left( \frac{|U|}{n} \right)^{1+\log(1-\biasname)/\log d}} +1 \right) \\
			&= 32d \cdot (1 +2\cdot 10^8  \cdot d^{4K}) \cdot  n \cdot \left( \frac{n}{|U|} \right)^{\log(1-\biasname)/\log d} + 8\sqrt{n}  \\
			&\overset{(\star)}{\leq}  10^{10} \cdot d^{4K+1}  \cdot  n \cdot \left( \frac{n}{|U|} \right)^{\log(1-\biasname)/\log d},
		\end{align*}where $(\star)$ follows since $K\geq 1$, $d\geq 3$ and $\log(1-\biasname)/\log d > \log(2/3)/\log 3  > -1/2$. 
		
		Finally, recall that $   \log(1+\Psi_G)\geq \Psi_G/2$ as $ 0<\Psi_G\leq 2 $, thus \[4K +1 = 4 \cdot \Big\lceil \frac{2}{\log(1 + \Psi_G)}\Big\rceil  + 1 \leq 4\cdot \left(\frac{4}{ \Psi_G}  +1 \right)  + 1 \leq  \frac{26}{\Psi_G},\] and so   $\Ex{\tau} \leq \kappa \cdot n \cdot \left(  n/|U| \right)^{\log(1-\biasname)/\log d}$ with $\kappa:= 10^{10} \cdot d^{26/\Psi_G}$ as claimed.
	\end{Proof}
	
	By a repeated application of the previous proposition, we can now establish the main result. That is, for any bounded-degree regular expander and any constant bias $\eps>0$, there exists a strategy for the $\eps$-$\tbrw$ which achieves an $O(n)$ cover time.  
	\begin{theorem}\label{thm:cover}
		Let $G$ be any connected $d$-regular graph and $n$ sufficiently large. Then for any $0 \leq \eps \leq 1 - e^{-\Psi_G/32}$  ,  
		\[
		\tetb(G) \leq C\cdot n \cdot \min\{\eps^{-1}, \; \log n \},
		\]
		where $C :=C(d,\Psi_G) =  10^{10} \cdot d^{50/\Psi_G} > 0$.
		
	\end{theorem}
	
	Note that for $\eps> 1- e^{- \Psi_G/32}$, we can apply the theorem for $\eps =  1- e^{- \Psi_G/32}$ and use monotonicity. Furthermore, for $\eps=0$ and $\Psi_G >0$ a fixed constant, the theorem above recovers the well-known cover time bound of $O(n \log n)$ for a simple random walk on bounded-degree expanders. \cref{thm:covercheap} follows directly from \cref{thm:cover}, since $1 - e^{-\Psi_G/32} \geq 1 - e^{-\gamma/64} $, and $d^{50/\Psi_G}\leq d^{100/\gamma}$ as $\Psi_G \geq \gamma/2 $ by \eqref{eq:vertexspectral}, and $1-e^{-\gamma/64} \geq   \gamma/128$ as $\gamma \in  (0, 1]$. 
	
	\begin{Proof}Observe that there are no connected graphs with $d=1$, and  if $d=2$ then $G$ must be a cycle. We can stochastically dominate the cover time of the biased random walk on a cycle by the time taken for a gambler to go bankrupt in the ``gamblers ruin'' process started from a fortune of $n$, where the gambler has a probability of $(1+\eps)/2$ of loosing each bet. It follows from classical results that this time is at most $2n/\eps$ in expectation \cite[Chapter~3.9]{GrimStir}, which satisfies the bound in the statement. Hence, we will assume $d \geq 3$ in the remainder of the proof.
		
		Recall that the cover time of the simple random walk on any regular graph is at most $2n^2 $ by \cite{Feige97}. Thus, we can assume from now on that $\Psi_{G} \geq 49\ln d/\ln n $; otherwise the bound in the statement, which is at least $2 d^{50/\Psi_G}n $ gives a worse bound on $ \tetb(G)$ than the $2n^2$ bound for the SRW on $G$. Hence, with these  assumptions, we can apply \cref{lem:keylemma_cover} in what follows.
		
		We divide the process into at most $\log_{2} n + 1$ phases, and in each phase we will at least halve the unvisited set. As soon as the unvisited set has been halved, we proceed to the next phase and apply \cref{lem:keylemma_cover} (where $\biasname=\eps$) with an updated unvisited set $U$ and, correspondingly, an updated walk $Q:=Q(U,\biasname)$. The expected number of time-steps until the unvisited set has been covered completely is therefore upper bounded by 
		\begin{align}
			\sum_{k=0}^{\log_2 n} 	\kappa \cdot \left( \frac{n}{n\cdot 2^{-k}} \right)^{\log(1-\biasname)/\log d} \cdot n &= 	\kappa \cdot  n\cdot  \sum_{k=0}^{\log_2 n}2^{k\cdot \log(1-\biasname)/\log d}\notag \\
			&\leq \kappa  \cdot n \cdot \min\left( \frac{1}{1-2^{\log(1-\biasname)/\log d}}, \log_2 n + 1  \right).  \label{eq:geosum}
		\end{align}Recall the following two well-known bounds: $\log(1+z) \leq z$ for all $z>-1$,   also $e^{-z} \leq  1 - z/2$ for all $z \in  [0, 1.59]$. We will now utilise them to simplify the third term of \eqref{eq:geosum}, giving 
		\begin{equation}\label{eq:simple}
			\frac{1}{1-2^{\log(1-\eps)/\log d}}
			\leq
			\frac{1}{1-2^{-\eps/\log d}}
			=
			\frac{1}{1-e^{-\eps\cdot \log 2 / \log d}}
			\leq
			\frac{4\log d}{\eps}.
		\end{equation}
		Recall that $\kappa =10^{10} \cdot d^{26/\Psi_G} > 0$ from \cref{lem:keylemma_cover} and observe that $\log d \leq d^{1/\Psi_G}$ holds as $\Psi_G \leq 2$ for any graph $G$ and $d\geq 3$.
		Thus by \eqref{eq:geosum} and \eqref{eq:simple} and as $n$ is large, 
		the cover time is at most $\kappa'\cdot n \cdot \min\{1/\eps, \; \log n \} $, where $\kappa'= 4d^{1/\Psi_G}\cdot 10^{10} \cdot d^{26/\Psi_G}\leq  10^{10} \cdot d^{50/\Psi_G}:= C$. 
	\end{Proof}
	
	\section{Lower Bounds on Boosting and the Cover Times of the Time-Biased Walk}
	\label{Boost}
	
	The aim of this section is to prove \cref{nonregbound}, which gives an upper bound on how much the $\eps$-$\tbrw$ can ``boost'' the probability of any event over the corresponding probability under  the simple random walk. This result is then used to prove \cref{thm:coveringlower}, which gives a super-linear lower bound on the cover time of the $\eps$-$\tbrw$ on any $d$-regular graph whenever $\eps =o(1/\log^2 d)$.

	\subsection{Lower Bound on ``Boosting''} 
	
	We introduce some essential definition before stating \cref{nonregbound},  which gives an upper bound on how much the bias random walk can ``boost'' the probabilities of events over the SRW. 
	
	Fix a vertex $u$, a non-negative integer $t$ and a set $S$ of $t$-step trajectories. Write $p_{u,S}$ for the probability that running a $\srw$ starting from $u$ for $t$ steps results in a trajectory  in $S$. Let $q_{u,S}(\eps)$ be the corresponding probability under the law of the $\eps$-$\tbrw$, which depends on the particular strategy used. We prove the following result relating $q_{u,S}(\eps)$ to $p_{u,S}$; recall that $d_{\max}$ denotes the maximum degree of $G$.
	
	\begin{theorem}\label{nonregbound}
		Let $G$ be any graph, $u\in V$, $S$ be a set of trajectories of length $t>0$ starting from $u$, and $0\leq  \eps \leq  1$. Then, for any strategy taken by the $\eps$-$\tbrw$, we have
		\[
		q_{u,S}(\eps) \leq \bigl( 1+ \eps(\dmax-1) \bigr)^t \cdot p_{u,S} .
		\]
		Further, if $\eps \leq 1/\dmax^{2\eta}$ for any $ 0<\eta \leq 1$, then 
		\[
		q_{u,S}(\eps)
		\leq
		\exp\bigl( 4 t / \dmin^{\eta} \bigr) \cdot (p_{u,S})^{\eta/(1+\eta)}.
		\] 
	\end{theorem}
	
	In order to prove this theorem we make use of the tree gadget introduced in \cite{POTC,ETB}. This is an encoding of trajectories of length at most $t$ from $u$ in a rooted graph $(G,u)$ by vertices of an arborescence $(\mathcal{T}_t,\mathbf{r})$, i.e., a tree with all edges oriented away from the root $\mathbf{r}$. Here we use bold characters to denote trajectories, and $\mathbf r$ will be the length-$0$ trajectory consisting of the single vertex $u$. The tree $\mathcal T_t$ consists of one node for each trajectory of length $i\leq t$ starting at $u$, and has an edge from $\mathbf{x}$ to $\mathbf{y}$ if $\mathbf{x}$ may be obtained from $\mathbf{y}$ by deleting the final vertex.
	
	We write $\Gamma^+(\mathbf{x})$ for the set of children of $\mathbf x$ in $\mathcal{T}_t$, and let $d^+(\mathbf{x}):=|\Gamma^+(\mathbf{x})|$. We denote the length of the trajectory $\mathbf x$ by $\abs{\mathbf{x}}$. Let $p_{\mathbf x,S}$ be the probability that, conditional on a partial trajectory $\mathbf x$, extending $\mathbf{x}$ to a trajectory of length $t$ under the law of the SRW results in an element of $S$. Let $q_{\mathbf x,S}(\eps)$ denote the analogous quantity for the $\eps$-$\tbrw$. Thus, $p_{u,S}$ and $q_{u,S}(\eps)$ correspond to the case $\mathbf{x}=u$. Additionally, let $W_u(k):=(X_i)_{i=0}^k$ be the trajectory of a simple random walk $X_i$ on $G$ up to time $k$, with $X_0=u$.

	For $r\in \mathbb R\setminus\{0\}$ and $d \in \mathbb{N}\setminus \{0\}$, the $r$-power mean $M_r$ of non-negative reals $v_1,\ldots,v_d$ is defined by \[
	M_r(v_1,\ldots,v_d):=\left(\frac{v_1^r+\cdots+v_d^r}{d}\right)^{1/r},
	\]
	and 
	\[M_{\infty}(v_1,\ldots,v_d)
	:=\max\{v_1,\ldots,v_d\}=\lim_{r\to \infty}M_r(v_1,\ldots,v_d).
	\]

	Recall that a strategy, applied to a specific time step, for the time biased random walk can be thought of as a function $\mathbf{b}$ from the history of the walk to probability distributions over the set of neighbours of the current vertex. The walk then uses this probability vector $\mathbf{b}$ if it gets to choose the next step. We introduce the biased operator $	\operatorname{B}_{\eps,\mathbf{b}}$ which, given a bias vector $\mathbf{b}$, determined by the strategy for the current trajectory, encodes one step of the process. That is, for $d\geq 1$, a vector $\mathbf{v}=(v_1, \dots, v_d)\in [0,\infty)^d$ and a probability vector $\mathbf{b}=(b_1,\dots, b_d)$, we let 
	\begin{equation*}
		\operatorname{B}_{\eps,\mathbf{b}}\left(\mathbf{v} \right)  := \sum_{i=1}^d \left(\frac{1 - \eps }{d} + \eps b_i\right)  \cdot  v_i .
	\end{equation*}

	Recall H\"older's inequality, which states that for any real $y_1,\ldots,y_d$ and $z_1,\ldots,z_d$, and for any $r,s\geq 1$ satisfying $1/r+1/s=1$, we have 
	\begin{equation}\label{holder} \sum_{i=1}^d y_i\cdot z_i\leq\Bigl(\sum_{i=1}^d y_i^r\Bigr)^{1/r}\cdot \Bigl(\sum_{i=1}^d z_i^s\Bigr)^{1/s}.\end{equation}

	The following lemma bounds the biased operator by the average times a factor depending on the length of the input vector and the bias. 
	\begin{lemma}\label{lem:conv}
		Let $d \in \mathbb{N} \setminus \{0\}$, $\eps\in [0,1] $, $\mathbf{v}\in [0,\infty)^d$, and $\mathbf{b} \in [0,1]^d$ be a probability vector. Then,  
		\[ \operatorname{B}_{\eps,\mathbf{b}}\left(\mathbf{v}  \right)\leq  \left(1 +   \eps (d-1)\right) \cdot M_{1}(\mathbf{v})  . \]
		Further, if $\eps \leq 1/d^{2\eta}$ for any $0 < \eta \leq 1$, then
		\[ \operatorname{B}_{\eps,\mathbf{b}}\left(\mathbf{v}  \right)\leq \exp\bigl(4 / d^{\eta}\bigr) \cdot  M_{(1+\eta)/\eta}(\mathbf{v}). \] 
	\end{lemma}
	\begin{Proof} Let $\mathbf{v}= (v_1,\dots, v_d)$. The first claim follows by the definition of  $\operatorname{B}_{\eps,\mathbf{b}}$, as
		\begin{align*}		\operatorname{B}_{\eps,\mathbf{b}}\left(\mathbf{v}\right)  &= \sum_{i=1}^d \left(\frac{1 - \eps }{d} + \eps b_i\right)  \cdot v_i \leq \left(1 - \eps  + \eps d\right) \cdot \sum_{i=1}^d   \frac{v_i}{d}=  \left(1 +   \eps (d-1)\right) \cdot M_1(\mathbf{v}). 
		\end{align*}For the second claim, for any $\mathbf{b}$ and any $1/r+1/s=1$ applying H\"older's inequality \eqref{holder} gives
		\begin{align}\operatorname{B}_{\eps,\mathbf{b}}\left(\mathbf{v}  \right)&=    \sum_{i=1}^d  \frac{1-\eps + d\eps b_i }{d^{1-1/s} } \cdot \frac{v_{i}}{d^{1/s} } \notag \\
			&\leq \left(\sum_{i=1}^d  \left(\frac{1-\eps + d\eps b_i }{d^{1-1/s} }\right)^r \right)^{1/r} \cdot M_{s}(\mathbf{v} )\notag \\
			&= \left(\frac{1}{d}\sum_{i=1}^d  \left(1-\eps + d\eps b_i\right)^r \right)^{1/r} \cdot M_{s}(\mathbf{v}). \label{eq:lastineq}
		\end{align} The first term in \eqref{eq:lastineq} is $M_r(\mathbf{u})$, where the vector $\mathbf{u}=(1-\eps +d\eps b_i)_{i=1}^d$ depends on $\mathbf{b}$. Observe that the $d$-dimensional vector $(1-\eps +d\eps,1-\eps,\dots, 1-\eps)$  majorizes $\mathbf{u}$ for any   choice of $\mathbf{b}$. Thus, since the $r$-power mean is Schur-convex by Lemma \ref{lem:powermeanschur}, we have  \begin{align}\label{eq:holdertoq}
			\operatorname{B}_{\eps,\mathbf{b}}\left(\mathbf{v}   \right) &\leq  \left(\frac{d-1}{d}\left(1-\eps\right)^r + \frac{1}{d}\left(1-\eps + d\eps  \right)^r \right)^{1/r} \cdot M_{s}(\mathbf{v})\notag \\
			&\leq  \left(1+ \frac{\left(1  + d\eps  \right)^r}{d} \right)^{1/r} \cdot M_{s}(\mathbf{v}).  \end{align} 
		Let $\eta \in (0,1]$ and $\eps \leq d^{-2\eta}$. Set $r:=1+\eta$. The first term of \eqref{eq:holdertoq} then satisfies
		\[ 
		\left(1+ \frac{\left(1  + d\eps  \right)^r}{d} \right)^{1/r} \leq 	\left(1+ \frac{\left(1  + d^{1-2\eta}  \right)^r}{d} \right)^{1/r}
		\le
		\left(1+ \frac{2^r  d^{(1-2\eta)r}}{d} \right)^{1/r}.\]
		Using the inequality $(1-2\eta)r = 1 -\eta -2\eta^2  \leq 1 -\eta  $ in the above gives   
		\[ \left(1+ \frac{\left(1  + d\eps  \right)^r}{d} \right)^{1/r}  \le  \left(1+ \frac{2^{1+\eta}  d^{1-\eta}}{d} \right)^{1/r}\leq 
		1+ \frac{2^{1+\eta}  d^{1-\eta}}{d} \le
		1+ 4d^{-\eta}\le
		e^{4d^{-\eta}}.
		\]
		Thus, the result holds by \eqref{eq:holdertoq} since $s=(1-1/r)^{-1} = r/(r-1)   = (1+\eta)/\eta$. 
	\end{Proof}
	
	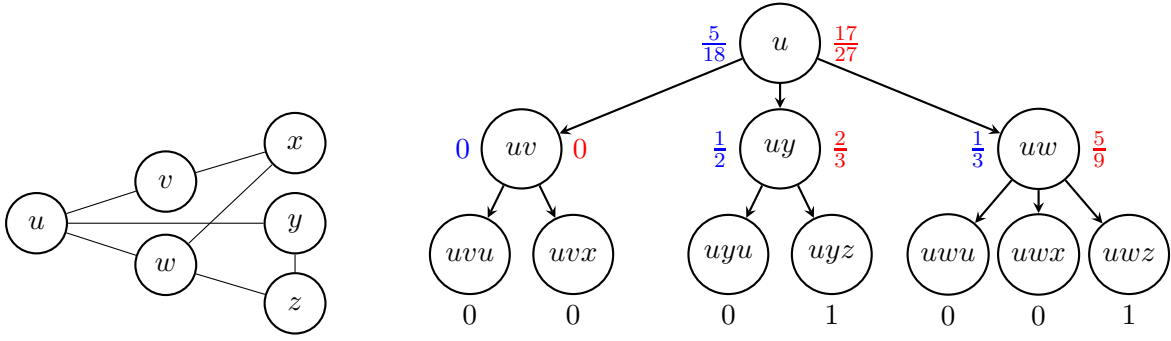
\begin{figure} 
		\begin{subfigure}{.35\textwidth}
			\begin{tikzpicture}[xscale=.85,yscale=0.85,knoten/.style={thick,circle,draw=black,minimum size=.8cm,fill=white},wknoten/.style={thick,circle,draw=black,minimum size=.6cm,fill=white},edge/.style={black},dedge/.style={thick,black,-stealth}]
				\node[knoten] (u) at (2,2) {$u$};
				\node[knoten] (v) at (4,2.65) {$v$};
				\node[knoten] (w) at (4,1.35) {$w$};
				\node[knoten] (x) at (6,3.25) {$x$};
				\node[knoten] (y) at (6,2) {$y$};
				\node[knoten] (z) at (6,0.75) {$z$};
				\draw[edge] (u) to (v);
				\draw[edge] (u) to (y);
				\draw[edge] (w) to (x);
				\draw[edge] (w) to (z);
				\draw[edge] (v) to (x);
				\draw[edge] (z) to (y);
				\draw[edge] (u) to (w);
			\end{tikzpicture}
		\end{subfigure}%
		\begin{subfigure}{.65\textwidth}
			\begin{tikzpicture}[xscale=0.85,yscale=0.7,knoten/.style={thick,circle,draw=black,minimum size=1.05cm,fill=white},wknoten/.style={thick,circle,draw=black,minimum size=1.05cm,fill=white},edge/.style={black},dedge/.style={thick,black,-stealth}]
				\node[knoten] (u1) at (2,6) [label=left:$\textcolor{blue}{\frac{5}{18}}$,label=right:$\textcolor{red}{\frac{17}{27}}$]{$\phantom{x}u\phantom{x}$};
				\node[knoten] (v2) at (-2,4)[label=left:$\textcolor{blue}{0}$,label=right:$\textcolor{red}{0}$] {$\;uv\;$};
				\node[knoten] (y2) at (2,4)[label=left:$\textcolor{blue}{\frac{1}{2}}$,label=right:$\textcolor{red}{\frac{2}{3}}$] {$\;uy\;$};
				\node[knoten] (w2) at (6,4)[label=left:$\textcolor{blue}{\frac{1}{3}}$,label=right:$\textcolor{red}{\frac{5}{9}}$] {$\;uw\;$};
				
				\node[wknoten] (u3) at (-2.8,2)[label=below:$0$] {$uvu$};
				\node[wknoten] (x3) at (-1.2,2)[label=below:$0$] {$uvx$};
				
				\node[wknoten] (u31) at (1.2,2)[label=below:$0$] {$uyu$};
				\node[wknoten] (z31) at (2.8,2)[label=below:$1$] {$uyz$};
				
				\node[wknoten] (u32) at (4.6,2) [label=below:$0$]{$uwu$};
				\node[wknoten] (x32) at (6,2) [label=below:$0$]{$uwx$};
				\node[wknoten] (z32) at (7.4,2) [label=below:$1$]{$uwz$};

				\draw[dedge] (u1) to (v2);
				\draw[dedge] (u1) to (y2);
				\draw[dedge] (u1) to (w2);
				
				\draw[dedge] (v2) to (u3);
				\draw[dedge] (v2) to (x3); 
				
				\draw[dedge] (y2) to (u31);
				\draw[dedge] (y2) to (z31);
				
				\draw[dedge] (w2) to (u32);
				\draw[dedge] (w2) to (x32);
				\draw[dedge] (w2) to (z32);
				
			\end{tikzpicture}
		\end{subfigure}
		\caption{Illustration of a (non-lazy) walk on a non-regular graph (shown on the lest) starting from $u$ with the objective of having visited $z$ by time $t=2$, this corresponds to a set of trajectories $S=\{uyz,uwz\}$. The conditional probabilities of achieving this, from the given node/trajectory, are given in blue (left) for the SRW, and in red (right) for the $\frac{1}{3}$-TBRW. The strategy used by the $\frac{1}{3}$-TBRW in this example is one that biases towards a vertex maximising the chance of reaching the target.}\label{fig:trajtree}
	\end{figure}
	With this we are now ready to prove the main result of this section. 
	\begin{Proof}[Proof of Theorem \ref{nonregbound}]For convenience we suppress the notational dependence of $q_{\mathbf{x},S}(\eps)$ on $\eps$. Recall that a strategy for the $\eps$-$\tbrw$ constitutes for each trajectory $\mathbf{x}$, a probability distribution $\mathbf{b}(\mathbf{x})=(b_\mathbf{y})_{\mathbf{y}\in \Gamma^+(\mathbf{x})}$ over the neighbours of $\mathbf{x}$ in $\mathcal{T}_t$. Given a strategy for the $\eps$-$\tbrw$, we assign to each node $\mathbf{x}$ of the tree gadget $\ct_t$ the value $q_{\mathbf{x},S} $ under $\mathbf{b}$. This is well-defined because  the values $q_{\mathbf{x},S}$ can be computed ``bottom up'' starting at the leaves, where if $\mathbf{x} \in V(\ct_t)$ is a leaf then $q_{\mathbf{x},S} $ is $1$ if $\mathbf x\in S$ and $0$ otherwise.

		Thus, if $\mathbf{x}$ is not a leaf, then the probability $q_{\mathbf{x},S}$ is the biased operator applied to the child $\mathbf{y}\in \Gamma^+(\mathbf{x})$ of $\mathbf{x}$, that is  \begin{equation}\label{qqqq}q_{\mathbf{x},S} =\operatorname{B}_{\eps,\mathbf{b}(\mathbf{x})}\left( \left(q_{\mathbf{y},S}\right)_{\mathbf{y}\in \Gamma^+(\mathbf{x})} \right).\end{equation} 
		
		Using \eqref{qqqq} one can calculate the conditional success probabilities of any strategy using the trajectory tree in a  ``bottom up'' fashion starting at the leaves, see \cref{fig:trajtree} for an example. 
		
		In general if we are given a trajectory $\mathbf{x}$ of length $i$, then we label vertices in $\mathbf{x}$ as $x_0,\dots,x_i$, thus $d^+(\mathbf{x})=d(x_i)$. In the proof all trajectories start from $u$, that is $x_0=u$.
		
		We first establish the second claim of this theorem, that is 
		\begin{equation} \label{eq:secclaim}
			q_{u,S}(\eps)
			\leq
			\exp\bigl( 4 t / \dmin^{\eta} \bigr) \cdot (p_{u,S})^{\eta/(1+\eta)}
			\Qwith
			\eps \leq  1/\dmax^{2\eta}
			\Qand
			0 < \eta \leq 1,
		\end{equation}
		since this is the more complicated one. We then explain how to adapt the proof to prove~the~first. 
		
		In order to prove \eqref{eq:secclaim}, we define the potential function $\Upsilon^{(i)}$ on the $i$-th generation of the  tree gadget $\ct$. Set $s:= (1+\eta)/\eta$, $\Upsilon^{(0)} := (q_{u,S})^s$, and for $i\geq 1$, define  
		\begin{equation}\label{Phi}\Upsilon^{(i)} :=  \sum\limits_{\mathbf y : |\mathbf{y}|=i} \prod_{j=0}^{i-1}\kappa(y_j)^s \cdot  (q_{\mathbf{y},S})^s\cdot \Pr{W_u(i) = \mathbf{y}},   \end{equation} where  $\kappa(y_j):=\exp(4 / d(y_j)^{\eta})$. For any trajectories $\mathbf{x},\mathbf{y}$ with $\mathbf{y} \in \Gamma^+(\mathbf{x})$, we have  
		\[\Pr{W_u(|\mathbf{y}|) = \mathbf{y}} = \frac{\Pr{W_u(|\mathbf{x}|) = \mathbf{x}}}{d^+(\mathbf{x})} ,\] thus, since each $\mathbf{y}$ with $|\mathbf{y}|=i$ has exactly one parent $\mathbf{x}$ with $|\mathbf{x}|=i-1$, we can write 
		\begin{equation}\label{PPhi}\Upsilon^{(i)} = \sum\limits_{\mathbf x :|\mathbf{x}|=i-1}\prod_{j=0}^{i-2}\kappa(x_j)^s\cdot  \sum_{\mathbf{y} \in \Gamma^+(\mathbf{x})}\kappa(x_{i-1})^s \cdot (q_{\mathbf{y},S})^s \cdot \frac{\Pr{W_u(i-1) = \mathbf{x}}}{d^+(\mathbf{x}) }.\end{equation} 
		We now show that $\Upsilon^{(i)} $ is non-decreasing in $i$ for $0 \leq i\leq t$. By combining \eqref{Phi} and \eqref{PPhi}, and defining an empty product to be $1$, we have for any $1 \leq i \leq t$,
		
		\begin{align}\label{eq:difference}
			\lefteqn{ \Upsilon^{(i-1)}-\Upsilon^{(i)}} \notag \\  &= \sum\limits_{\mathbf x: |\mathbf{x}|=i-1} \prod_{j=0}^{i-2}\kappa(x_j)^s  \left((q_{\mathbf{x},S})^s - \kappa(x_{i-1})^s \sum_{\mathbf{y} \in \Gamma^+(\mathbf{x})}\frac{(q_{\mathbf{y},S})^s}{d^+(\mathbf{x})} \right) \Pr{W_u(i-1) = \mathbf{x}}.
		\end{align}
		However, since $\mathbf{x}$ is not a leaf, by \eqref{qqqq} and the second claim of Lemma \ref{lem:conv} we obtain   \begin{align}\label{eq:drop} (q_{\mathbf{x},S})^s &=  \operatorname{B}_{\eps, \mathbf{b}(\mathbf{x})}\Big(\left( q_{\mathbf{y},S}\right)_{\mathbf{y} \in \Gamma^+(\mathbf{x})} \Big)^s \notag \\  
			&\leq 
			\left(
			\exp\bigl(4 / d^+(\mathbf{x})^{\eta} \bigr)
			\cdot
			M_{s}\bigl( ( q_{\mathbf{y},S})_{\mathbf{y} \in \Gamma^+(\mathbf{x})} \bigr)
			\right)^s \notag \\
			&= \kappa(x_{i-1})^s \cdot \sum_{\mathbf{y}\in \Gamma^+(\mathbf{x})}\frac{(q_{\mathbf{y},S})^s}{d^+(\mathbf{x})}. \end{align}   	Thus, by \eqref{eq:difference} and the above, we see that $\Upsilon^{(i)} $ is non-decreasing in $i$ for $1 \leq i\leq t$. 
		Observe that, for any $\mathbf{y}$ with $|\mathbf{y}|=t$ and any integer $0\leq j\leq t-1$ we have 
		\[\kappa(y_j)=\exp\bigl( 4 / d(y_j)^{\eta} \bigr) \leq  \exp\bigl( 4 / (\dmin)^{\eta} \bigr). \] 
		Additionally, if $|\mathbf{y}|=t$ then $q_{\mathbf{y},S}=1 $ if $\mathbf{y} \in S$ and $0$ otherwise. Thus,  
		\begin{align*}\Upsilon^{(t)} &= \sum_{\mathbf{y} \colon |\mathbf{y}|=t} \prod_{j=0}^{t-1}\kappa(y_j)^s\cdot (q_{\mathbf{y},S})^s\cdot \Pr{W_u(t) = \mathbf{y}}\\
			&\leq \exp\bigl( 4 s t / \dmin^{\eta} \bigr) \cdot  \sum_{\mathbf{y} \colon |\mathbf{y}|=t} (q_{\mathbf{y},S})^s\cdot \Pr{W_u(t) = \mathbf{y}}\\
			&\leq \exp\bigl( 4 s t / \dmin^{\eta} \bigr) \cdot  \sum_{\mathbf{y} \colon |\mathbf{y}|=t} \mathbf{1}_{\mathbf{x}\in S}\cdot \Pr{W_u(t) = \mathbf{y}}\\ 
			&= \exp\bigl( 4 s t / \dmin^{\eta} \bigr) \cdot p_{u,S}. \end{align*} Finally, since  $\Upsilon^{(0)} = (q_{u,S})^s$ and $\Upsilon^{(i)}$ is non-decreasing in $i$, \[(q_{u,S})^s = \Upsilon^{(0)}\leq \Upsilon^{(t)} = \exp\bigl( 4 s t / \dmin^{\eta} \bigr) \cdot p_{u,S},  \] and taking $s$-th roots of both sides gives the second claim as $s= (1+\eta)/\eta$. 
		
		The proof of the first claim follows the exact same strategy, we just use the simpler potential 
		\begin{equation*}\Upsilon^{(i)}=  \sum\limits_{|\mathbf{y}|=i} \kappa^i \cdot  q_{\mathbf{y},S}\cdot \Pr{W_u(i) = \mathbf{y}},   \end{equation*} where $\kappa:= 1+\eps(\dmax -1)$, and we apply the first claim of Lemma \ref{lem:conv} to establish the potential drop corresponding to \eqref{eq:drop}, as opposed to the second. 
	\end{Proof}

	\subsection{Lower Bound on Cover Time}
	
	We close the paper by deriving \cref{thm:coveringlower}, which we restate for the readers' convenience.

	\covlower*
	
	In order to prove \cref{thm:coveringlower} we will need the following theorem which shows that linear cover time for the simple random walk on any graph is exponentially unlikely. This generalized earlier results for bounded-degree graphs \cite{LinCovNo} and trees \cite{Yehudayoff12}.
	
	\begin{theorem}[{\cite[Theorem 1.1]{dubroff2021linear}}]\label{thm:dubkahn}
		For any constant $C>0$, there
		exists a constant $\alpha:=\alpha(C) > 0$
		such that, for any $n$-vertex graph $G$, if $(X_i)_{i\geq 0}$ is the simple random walk on $G$ with any starting distribution $\mu$, then
		\[ \Pr{\{X_0,X_1,\dots , X_{Cn}\} = V \, \mid\, X_0\sim \mu  } < e^{-\alpha \cdot n}.\]
	\end{theorem}
	
	We can now establish  \cref{thm:coveringlower} using  \cref{nonregbound,thm:dubkahn}.

	\begin{Proof}[Proof of \cref{thm:coveringlower}] Note that since a $1$-regular graph is not connected, we can assume that $d\geq 2$. Let $S$ be all trajectories of length $t:= 3Cn$   which cover the graph, where $C \geq 1$ is a constant. Thus, if the trajectory of any walk of length $t$ is not in $S$, then there is at least one vertex that has not been visited.  
		
		Now, by \cref{thm:dubkahn} there exists a constant $\alpha > 0$ such that \begin{equation}\label{eq:bddonSRW} p_{u,S}\leq e^{-\alpha \cdot n}.
		\end{equation} 
		Note that $\alpha$ depends only on $C$, and we may assume in the remainder of the proof that $\alpha \leq 1$. We will proceed by a case distinction depending on whether $d$ is small or large. As the threshold, for $\alpha:=\alpha(C)$ set \begin{equation} 
			\Delta := \Delta(C)=   \exp\Big(\exp\Big(\frac{24C+2}{\alpha}\Big)\Big) . \end{equation} 
		
		\textit{Case 1.} Suppose first that $2 \leq d \leq \Delta$. By the first claim in Theorem \ref{nonregbound}, for any $  \eps\in [0,1]$ and  any strategy taken by the $\eps$-$\tbrw$, 
		\[q_{u,S}(\eps) \leq (1+\eps(d-1))^t\cdot p_{u,S} \leq e^{\eps\cdot d\cdot t}\cdot p_{u,S} . \] 
		It follows that if we take $c_1:=\frac{\alpha}{6C}$, then for any $\eps \leq c_1/d$ we have 
		\begin{equation}\label{eq:smalldprob} q_{u,S}(\eps) \leq e^{\eps\cdot d\cdot t}\cdot p_{u,S} \leq  e^{\frac{\alpha}{6Cd} \cdot d\cdot 3Cn  } \cdot  e^{-\alpha \cdot n} = e^{-\alpha \cdot n / 2}.   \end{equation}   	 
		
		\textit{Case 2.} Suppose now that $d > \Delta$. To begin, as $d\geq \Delta \geq 3$, we  set \[ \eta := \frac{\log\log d}{\log d} \in
		(0,1 ],\]  then we have $d^{-2\eta} = d^{-\frac{2\log\log d}{\log d}} = (\log d )^{-2} $ and $\frac{\eta }{1+\eta} \geq  \frac{\log\log d}{2\log d}$. Thus, by \eqref{eq:bddonSRW} and applying the second claim in Theorem \ref{nonregbound}, for any strategy taken by the $\eps$-$\tbrw$ with $\eps=d^{-2\eta}$, we have
		\begin{align} \label{eq:largedprob}
			q_{u,S}(\eps) 
			&
			\leq
			\exp\bigl( 4 t d^{-\eta} \bigr) \cdot (p_{u,S})^{\frac{\eta }{1+\eta}}\notag
			\\&
			\leq
			\exp\bigl( 4 \cdot 3 C n \cdot (\log d)^{-1} \bigr) \cdot \exp(-\alpha n)^{\frac{\log\log d}{2\log d} }\notag
			\\&
			=
			\exp\left(\left(\frac{12C}{ \log d }-\alpha\cdot \frac{\log\log d}{2\log d}\right)\cdot n\right)\notag
			\\&
			\leq
			\exp\left(-  \frac{n}{\log d }\right), 
		\end{align} 
		where the final inequality holds since $d> \Delta :=\exp(\exp(\frac{24C+2}{\alpha}))$. This concludes the case distinction.
		
		Combining both cases, i.e., \eqref{eq:smalldprob} and \eqref{eq:largedprob}, it follows that if $d \leq \Delta$ and $\eps\leq c_1/ d$, or $d > \Delta$ and $\eps \leq d^{-2\eta}= (\log d)^{-2}$, that $q_{u,S}(\eps) <1/2$, for suitably large $n$. Consequently, 
		\begin{equation}\label{eq:covlower}\tetb(G) \geq 3Cn\cdot (1- q_{u,S}(\eps)  ) >Cn. 
		\end{equation}
		
		It only remains to verify that in both cases
		$\epsilon \leq c/\log^2 d$. First set $c:= c_1/ (3\Delta)= \alpha/(18 \Delta C)$.
		Note that $c\leq 1$ since $\alpha \leq 1$, $C \geq 1$ and $\Delta > 1$. Following the theorem statement, let us take any $\eps$ satisfying $\eps \leq c  (\log d)^{-2}$. Then, for the case $2\leq d \leq \Delta$ we have $\eps \leq c  (\log d)^{-2} \leq 3c = c_1/\Delta \leq c_1/d$, thus \eqref{eq:covlower} holds in this case. For the remaining case $d>\Delta$, since we have $c\leq 1$, we have $\eps \leq    (\log d)^{-2}$,  thus \eqref{eq:covlower} holds in this case as well. 
	\end{Proof}

	\section{Conclusions and Open Problems}\label{sec:conc}
	In \cref{res:robust} we gave a lower bound on the edge conductance (and spectral gap) of a perturbed chain which is decreasing in $d$.
	
	\begin{pbl}Can \cref{res:robust} be improved so that the lower bound is independent of $d$?
	\end{pbl}
	
	Our second open problem is to enlarge the class of graphs for which a linear cover time is known for the $\eps$-$\tbrw$. 
	Recall that $\tetb(G)$ implicitly minimizes over all $\eps$-biased strategies.
	
	\begin{pbl}
		Is it true that for any fixed $\eps>0$ and any bounded-degree graph $G$ we have $\tetb(G) = \Theta(n)$?
	\end{pbl}
	There are of course many candidate graph and bias pairs for which the question above makes sense, however to us bounded degree with a fixed bias is a very interesting combination and thus a natural place to start. We have established this, with some effort, for bounded-degree regular expander graphs. However, it is unclear whether this can be extended to \emph{any} bounded-degree graph. Indeed, a weaker bound of $o(n \log n)$, or even $O(n \operatorname{polylog} n)$, is not known and would already be progress.
	
	More generally, \cref{thm:covercheap} shows that the cover time of an $\eps$-TBRW on a $d$-regular graph $G$ satisfies 
	\[\tetb(G)=O\left(d^{100/\gamma}\cdot n \cdot \min\{\eps^{-1}, \; \log n \} \right)   .\] 
	We can compare this with the previously best-known bound from \cite{ETB},
	\[	\tetb(G)=O\left(\frac{\thit}{\eps}\cdot \log\left( \frac{ d\cdot \log  n}{\gamma} \right)\right),\]
	where $\thit$ is the maximum hitting time of the SRW, which is $O(n)$ in expanders.
	The first bound has significantly worse dependence on the spectral gap $\gamma$ and degree $d$. However, the second is always $\Omega(n\log\log n)$
	since $\thit \ge n/4$. This leads us to the following open problem.
	
	\begin{pbl}
		What is the best bound on the cover time of a graph by the $\eps$-$\tbrw$ in terms of $n, \gamma$ and $d$? 
	\end{pbl}
	
	The spectral gap is a natural choice of parameter as it is so closely related to the mixing time (and thus cover time) of the simple random walk, however it could be the case that this is not the right parameter when considering the $\eps$-$\tbrw$. This leads us to the next open problem. 
	\begin{pbl}
		Is there a graph parameter which more crisply characterises the cover time of the $\eps$-$\tbrw$? 
	\end{pbl}
	\cref{thm:coveringlower} shows that for every $d$-regular graph a bias of  $\eps=\Omega(1/\log^2 d)$ is required to cover it in linear time with an $\eps$-$\tbrw$. It is likely that the square root is an artifact of our proof, and that the dependence on $d$ can be improved%
	---although, it is already pretty good.
	However, it is not at all clear whether the dependence on $d$ can be eliminated altogether.
	
	\begin{pbl}
		Does there exist a $d$-regular graph $G$ with $d \to \infty$ such that
		\[
		\tetb(G) = \Theta(n) \Quad{for some} \eps = o(1)?
		\]
	\end{pbl}

	\iftoggle{anonymous}{
	}{%
		\section*{Acknowledgments}
		We thank John Haslegrave, Peleg Michaeli, and Endre Cs\'oka for some insightful discussions. 
	}
	
	\bibliographystyle{abbrv}
	\bibliography{ref}

	\appendix

	\section{Majorization and Schur-Convexity}
	
	For two  vectors $\mathbf{x},  \mathbf{y} \in \mathbb {R}^{d}$, we say that $\mathbf {x}$ majorizes $\mathbf {y}$ if $\sum_{i=1}^d x_i = \sum_{i=1}^d y_i$ and 
	\[ 
	\sum _{i=1}^{k}x_{i}^{\downarrow }\geq \sum _{i=1}^{k}y_{i}^{\downarrow } \quad \text{for all }k=1,\dots ,d,\]
	where $x_{i}^{\downarrow }$ denotes the $i$-th largest entry of $\mathbf{x}$. For $X\subseteq \mathbb{R}$, a function $f : X^d \to \mathbb{R}$ is Schur-convex (see \cite[Chapter~3]{MRBook}) if for any  $\mathbf{x}, \mathbf{y} \in X^d$, if $x$ majorizes $y$ then $f(\mathbf{x}) \geq f(\mathbf{y})$.

	\begin{lemma}[{\cite[Proposition~C.1]{MRBook}}] \label{lem:sum_of_convex_is_schur_convex}
		Let $f$ be a convex function. Then $g(x_1, \ldots, x_d) := \sum_{i = 1}^d f(x_i)$ is Schur-convex. 
	\end{lemma} 
	
	\begin{lemma}[{\cite[Chapter~3.B.1.b]{MRBook}}] \label{lem:increasing_schur_convex}
		An increasing function of a Schur-convex function is Schur-
		convex. 
	\end{lemma} 
	Using these facts we can show the following lemma. 
	\begin{lemma}\label{lem:powermeanschur}
		For any $r\geq 1$, the  $r$-power mean $M_{r}(\cdot)$ is Schur-convex on non-negative real vectors.   
	\end{lemma}
	\begin{proof}
		Since the function $f(x)=x^r/d$ is convex on $\mathbb{R}_{\geq 0}$ for $r\geq 1$, it follows that  $g(x_1, \dots, x_d)=\frac{1}{d}\sum_{i=1}^d x_i^r  $ is Schur-convex on non-negative real vectors by Lemma \ref{lem:sum_of_convex_is_schur_convex}. The result now follows by Lemma \ref{lem:increasing_schur_convex} since $h(x)= x^{1/r}$ is an increasing function on $\mathbb{R}_{\geq 0}$.
	\end{proof}
	
\end{document}